\documentclass[12pt,reqno,intlimits,twoside]{amsart}
\usepackage[cp1251]{inputenc}
\usepackage[T2A]{fontenc}
\usepackage[english]{babel}
\inputencoding{cp1251}
\usepackage{amsfonts,amsmath,amsxtra,amsthm,amssymb,latexsym,euscript}
\usepackage{graphicx}

\newtheorem{thm}{Theorem}[section]
\newtheorem{cor}[thm]{Corollary}
\newtheorem{lem}[thm]{Lemma}
\newtheorem{prop}[thm]{Proposition}
\theoremstyle{definition}
\newtheorem{defn}{Definition}
\theoremstyle{remark}
\newtheorem{rem}{Remark}
\newtheorem{eg}{Example}

\DeclareMathOperator*{\dom}{\mathrm{dom}}
\let\im\Im

\DeclareMathOperator*{\PLUS}{\oplus} \allowdisplaybreaks
\begin{document}
\allowdisplaybreaks
\title[Isospectrality for graph Laplacians]{Isospectrality for graph Laplacians under the change of coupling at graph vertices: necessary and sufficient conditions}
\author{Yulia Ershova}
\address{Institute of Mathematics, National Academy of Sciences of Ukraine. 01601 Ukraine, Kiev-4,
3, Tereschenkivska st.} \email{julija.ershova@gmail.com}
\author{Irina I. Karpenko}
\address{Department of Algebra and Functional Analysis, V.I. Vernadsky Taurida National university. 95007 Ukraine, Autonomous Republic of Crimea, Simferopol, 4 Vernadsky pr.}
\email{i\_karpenko@ukr.net}
\author{Alexander V. Kiselev}
\address{Department of Functional Analysis, Pidstryhach Institute for Applied Problems of Mechanics and Mathematics,
National Academy of Sciences of Ukraine, 3-b Naukova Str. 79060,
L'viv, Ukraine}
\address{Department of Higher Mathematics and Mathematical Physics,
St. Petersburg State University, 1 Ulianovskaya Street, St.
Petersburg, St. Peterhoff 198504 Russia}
\email{alexander.v.kiselev@gmail.com}
\thanks{The third authors' work was partially supported by the RFBR, grant no. 12-01-00215-a.}


\begin{abstract}
Laplace operators on finite compact metric graphs are considered
under the assumption that matching conditions at graph vertices
are of $\delta$ and $\delta'$ types. Assuming rational
independence of edge lengths, necessary and sufficient conditions
of isospectrality of two Laplacians defined on the same graph are
derived and scrutinized. It is proved that the spectrum of a graph
Laplacian uniquely determines matching conditions for ``almost
all'' graphs.
\end{abstract}
\maketitle

\section{Introduction}
In the present paper we study the so-called quantum graph, i.e., a
metric graph $\Gamma$  with an associated second-order
differential operator acting in Hilbert space $L^2(\Gamma)$ of
square summable functions with an additional assumption that
functions belonging to the domain of the operator are coupled by
certain matching conditions at graph vertices. These matching
conditions reflect the graph connectivity and usually are assumed
to guarantee self-adjointness of the operator. Recently these
operators have attracted a considerable interest of both
physicists and mathematicians due to a number of important
physical applications, e.g., to the study of quantum wavequides.
Extensive literature on the subject is surveyed in, e.g.,
\cite{Kuchment,Kuchment2}.

The present paper is devoted to the study of an inverse spectral
problem for Laplace  operators on compact metric graphs. The
inverse problem we have in mind is this: does the spectrum of a
graph Laplacian uniquely determine matching conditions at graph
vertices (and therefore, the operator itself)?

The related problem of whether the spectrum determines the graph
topology (in the case when matching conditions are assumed to be
standard, or Kirchhoff) has been studied extensively. To name just
a few, we would like to mention the pioneering works
\cite{Roth,Smil1,Smil2} and later contributions
\cite{Kura1,Kura2,Kura3,Kostrykin}. These papers utilize an
approach to the problem based on the so-called trace formula which
relates the spectrum of a quantum graph to the set of closed paths
on the underlying metric graph. Different approaches to the same
problem were developed, e.g., in
\cite{pivo,Belishev_tree,Belishev_cycles}, see also
\cite{NabokoKurasov} for the analysis of yet another related
problem.

On the other hand, the inverse problem we study in the present
paper has to the best of our knowledge surprisingly attracted much
less interest. We believe it was first treated in \cite{Carlson}.
In the cited paper the square of self-adjoint operator of the
first derivative was considered (thus disallowing both $\delta-$
and $\delta'-$ coupling at graph vertices) on a subset of metric
graphs. Then, after being mentioned in \cite{Kura2,Kura3}, it was
treated in \cite{Avdonin}, but only in the case of star graphs. In
our papers \cite{Yorzh1,Yorzh2,Yorzh3} we suggested an approach
based on the theory of boundary triples, leading to the asymptotic
analysis of Weyl-Titchmarsh $M$-function of the graph.

In the present paper we consider the case of a general connected
compact finite metric graph (in particular, this graph is allowed
to possess both cycles and loops), but only for two possible
classes of matching conditions at graph vertices. Namely, each
vertex is allowed to have matching of either $\delta$ or $\delta'$
type (see Section 2 for definitions). The named two classes
singled out by us prove to be physically viable \cite{Exner1,
Exner2}. We further restrict ourselves to the case when all edge
lengths of the graph $\Gamma$ are assumed to be rationally
independent. This case is known to yield the most complete results
in the inverse topology problem as well.

The results we present demonstrate that \emph{most} graphs prevent
non-trivial isospectral configurations. Moreover, isospectrality
turns out to be only possible in ``simple'' graphs (chains or
simple cycles). Nevertheless, the answer to the problem of
isospectrality is not trivial: in particular, we demonstrate that
there exists a stark difference between graph Laplacians with
matching conditions of a ``pure'' type (be it $\delta$ or
$\delta'$) and those with matching conditions of ``mixed'' type.

We point out that techniques developed in the present paper also
allow for the treatment of inverse topology problem; moreover,
consideration of Schr\"odinger operators on metric graphs is also
possible (see, e.g., \cite{Yorzh2,Yorzh3}). The corresponding
results will be published elsewhere.

The approach suggested is based on the celebrated theory of
boundary triples \cite{Gor}. The notion of a generalized
Weyl-Titchmarsh $M$-function for a properly chosen maximal
operator allows us to reduce the study of spectra of graph
Laplacians to the study of ``zeroes'' of the corresponding
finite-dimensional analytic matrix function. The results are then
obtained by asymptotic analysis.

The results of the present paper are in a sense complete, i.e.,
can be used to ascertain the absence of isospectral configurations
or otherwise for \emph{any} compact metric graph; yet at the same
time a number of meaningful questions are left open.

The paper is organized as follows.

Section 2 introduces the notation and contains a brief summary of
the material on the boundary triples used by us in the sequel as
well as an explicit formula (derived in \cite{Yorzh3}) for the
Weyl-Titchmarsh $M$-function written down in what we would like to
think of as its ``natural'' form.

In Section 3 we derive conditions necessary and sufficient for
isospectrality of a pair of graph Laplacians under the additional
restriction that all edge lengths of the graph are rationally
independent.

Finally, in Section 4 we justify the procedure of reducing the
problem of isospectrality to the same for a pair of Laplacians
defined on a ``smaller'' graph which is nothing but the original
one, ``trimmed'' in a proper way.

In our view, the main results of the paper are contained in the
following assertions:
\begin{itemize}
\item Theorem \ref{Thm_n_and_s}, which, although unsuitable for
most applications, serves not only as a cornerstone of our further
analysis, but also allows to give complete solution to the problem
of isospectrality for \emph{sufficiently simple} graphs. \item
Theorem \ref{thm_uniqueness_nonzero} ascertains that any
non-trivial graph with all edges (except loops) of mixed type,
i.e., with any two adjacent vertices having different ($\delta$
and $\delta'$) types of coupling, disallows isospectral
configurations of coupling constants provided that the latter are
all assumed to be non-zero. \item Theorem \ref{Thm_bloody} then
provides the ultimate answer to the problem of isospectrality for
graphs with all edges of mixed type, stating that if the graph is
clean \cite{Aharonov} and the coupling constants \emph{are}
allowed to zero out, the only non-trivial graph permitting
isospectral configurations is the chain $A_3$. \item Theorem
\ref{thm_trimming} tells one that under the assumption that two
graph Laplacians are isospectral on a graph $\Gamma$, the
corresponding pair of Laplacians has to be isospectral on a
smaller (``trimmed'', see Definition \ref{def_trimming}) graph
$\Gamma^{(e)}$ with both sets of coupling constants changing in a
controllable way. This result shows that for an arbitrary graph
one can indeed reduce the study of isospectrality phenomenon to
the same on either an arbitrarily simple graph (if all graph
vertices have the same type of coupling) \emph{or} a graph with
all edges of mixed type, the answer in both cases being either
known or easy to obtain.

\item Finally, the balancing conditions \eqref{eq_balance_general}
of Theorem \ref{thm_uniqueness_nonzero} and the assertion of
Theorem \ref{Thm_uniqueness_zero} can be viewed as ``local''
uniqueness results. In particular, Theorem
\ref{Thm_uniqueness_zero} ascertains that on a clean graph
$\Gamma$ isospectrality of two graph Laplacians immediately yields
uniqueness of coupling constants at all vertices, adjacent to the
one where coupling constants of both Laplacians are known to zero
out simultaneously; the balancing conditions
\eqref{eq_balance_general} guarantee in a sense the converse,
allowing to ascertain the uniqueness of coupling constants at a
given vertex based on equality of coupling constants at adjacent
vertices.
\end{itemize}

\section{Preliminaries}

We call $\Gamma=\Gamma(\mathbf{E_\Gamma},\sigma)$ a finite compact
metric graph, if it is a collection of a finite non-empty set
$\mathbf{E_\Gamma}$ of compact intervals $e_j=[x_{2j-1},x_{2j}]$,
$j=1,2,\ldots, n$, called \emph{edges}, and of a partition
$\sigma$ of the set of endpoints $\{x_k\}_{k=1}^{2n}$ into $N$
classes, $\mathbf{V_\Gamma}=\bigcup^N_{m=1} V_m$. The equivalence
classes
 $V_m$, $m=1,2,\ldots,N$ will be called \emph{vertices} and the number of elements belonging to the set $V_m$ will be called the \emph{valence} of the vertex
$V_m$ (denoted $\deg V_m$).

With a finite compact metric graph $\Gamma$ we associate Hilbert
spaces
$$L_2(\Gamma)=\PLUS_{j=1}^n L_2(\Delta_j) \mbox{ and }\ W_2^2(\Gamma)=\oplus_{j=1}^n W_2^2(\Delta_j).$$
These spaces obviously do not ``feel'' the graph connectivity,
being the same for each graph with the same number of edges of
same lengths.

In what follows, we single out two natural \cite{Exner1} classes
of so-called \emph{matching conditions} which lead to a correctly
defined self-adjoint operator on the graph $\Gamma$, namely,
matching conditions of $\delta$ and $\delta'$ types. In order to
describe these, we introduce the following notation. For a
function $f\in W_2^2(\Gamma)$, we will use  the following
definition of the normal derivative $\partial_nf(x_{j})$ at the
endpoints of an interval $e_k$ throughout:
\begin{equation*}\label{fdnfx}
\partial_n f(x_j)=\left\{ \begin{array}{rl} f'(x_j),&\mbox{ if } x_j \mbox{ is the left endpoint of an edge},\\
-f'(x_j),&\mbox{ if } x_j \mbox{ is the right endpoint of an
edge.}
\end{array}\right.
\end{equation*}

Associate either of two symbols, $\delta$ or $\delta'$, to each
vertex of the graph $\Gamma$. The graph thus obtained will be
referred to as \textbf{marked} and denoted $\Gamma_{\delta}$. A
vertex $V$ will be called \emph{a vertex of $\delta$ type}, if $V$
is marked by the symbol $\delta$, and \emph{a vertex of $\delta'$
type} in the opposite case. Any marked graph $\Gamma_{\delta}$
determines the lineal
$$
\mathcal{D}(\Gamma_{\delta}):=\left\{f\in W_2^2(\Gamma)\ \left|\
                                                 \begin{array}{l}
                                                   f\ \mbox{is continuous at all}\\ \mbox{internal vertices of }\ \delta \mbox{ type}, \\
                                                   \partial_nf(x_i)=\partial_nf(x_j) \forall i,j: x_i,x_j\in V \mbox{ at all}\\ \mbox{internal vertices $V$ of }\  \delta' \mbox{ type} \\
                                                 \end{array}                                          \right. \right\}.
$$
Note that the latter definition imposes no restrictions on the
functions from $\mathcal{D}(\Gamma_{\delta})$ at boundary vertices
of the graph, i.e., at vertices of valence 1. For reasons of
convenience, we refer to all graph vertices of higher valence as
\emph{internal vertices} throughout.

We are now all set to define the Laplacian $A_{\vec\alpha}$ on the
graph $\Gamma_{\delta}$ which is the operator of the negative
second derivative on functions from
$f\in\mathcal{D}(\Gamma_{\delta})$ subject to the following
additional \emph{matching conditions}.
\begin{itemize}
\item[\textbf{($\delta$)}] If $V_k$ is of $\delta$ type, then
$$
\sum_{x_j \in V_k} \partial _n f(x_j)={\alpha_k} f(V_k).
$$
\item[\textbf{($\delta'$)}] If $V_k$ is of $\delta'$ type, then
$$
\sum_{x_j \in V_k} f(x_j)=-{\alpha_k}\partial_n f(V_k).
$$
\end{itemize}
Here $\vec\alpha=(\alpha_1,\alpha_2,...\alpha_N)$ is a set of
arbitrary real constants which we will refer to as \emph{coupling
constants}, whereas $f(V_k)$ and $\partial_n f(V_k)$ are
well-defined on $\mathcal D(\Gamma_{\delta})$ at vertices of
$\delta$ and $\delta'$ type, respectively.

Provided that all coupling constants $\alpha_m$, $m=1\dots N$, are
real, the operator $A_{\vec{\alpha}}$ is self-adjoint in  Hilbert
space $L_2(\Gamma)$ \cite{Exner1,KostrykinS}. Throughout the
present paper, we are going to consider this self-adjoint
situation only, although it has to be noted that the approach
developed can be used for the purpose of analysis of the general
non-self-adjoint situation as well.

Clearly, the self-adjoint operator thus defined on a finite
compact metric graph has purely discrete spectrum that accumulates
to $+\infty$. In order to ascertain this, one only has to note
that the operator considered is a finite-dimensional perturbation
in the resolvent sense of the direct sum of Sturm-Liouville
operators on individual edges.

Note that w.l.o.g. each edge $e_j$ of the graph $\Gamma$ can be
considered to be an interval $[0,l_j]$, where
$l_j=x_{2j}-x_{2j-1}$, $j=1\dots n$ is the length of the
corresponding edge.  Throughout the present paper we will
therefore only consider this situation.

The analysis presented in the present paper is essentially based
on the theory of boundary triples \cite{Gor,Ko1,Koch,DM} applied
to the class of operators introduced above. Two fundamental
concepts of this theory which we will recall below are the
concepts of a boundary triple and of the Titchmarsh-Weyl
generalized matrix-function. Assume that $A_{\min}$ is a symmetric
densely defined operator in Hilbert space $H$, and that its
deficiency indices are equal. Put $A_{\max}:=A_{\min}^*$.

\begin{defn}[\cite{Gor,Ko1,DM}]\label{Def_BoundTrip}
Let $\Gamma_0,\ \Gamma_1$ be linear mappings of $\dom(A_{\max})$
to $\mathcal{H}$ which is a separable Hilbert space. The triple
$(\mathcal{H}, \Gamma_0,\Gamma_1)$ is called \emph{a boundary
triple} for the operator $A_{\max}$ if:
\begin{enumerate}
\item for all $f,g\in \dom(A_{\max})$
$$
(A_{\max} f,g)_H -(f, A_{\max} g)_H = (\Gamma_1 f, \Gamma_0
g)_{\mathcal{H}}-(\Gamma_0 f, \Gamma_1 g)_{\mathcal{H}}.
$$
\item the mapping $\gamma$ defined as $f\longmapsto (\Gamma_0 f;
\Gamma_1 f),$ $f\in \dom(A_{\max})$ is surjective, i.e., for all
$Y_0,Y_1\in\mathcal{H}$ there exists an element $y\in
\dom(A_{\max})$ such that $\Gamma_0 y=Y_0,\ \Gamma_1 y =Y_1.$
\end{enumerate}

A non-trivial extension ${A}_B$ of the operator $A_{\min}$ such
that $A_{\min}\subset  A_B\subset A_{\max}$  is called
\emph{almost solvable} if there exists a boundary triple
$(\mathcal{H}, \Gamma_0,\Gamma_1)$ for $A_{\max}$ and a bounded
linear operator $B$ defined on $\mathcal{H}$ such that for every
$f\in \dom(A_{\max})$
$$
f\in \dom({A_B})\text{ if and only if } \Gamma_1 f=B\Gamma_0 f.
$$

The operator-function $M(\lambda),$ defined by
\begin{equation*}\label{Eq_Func_Weyl}
M(\lambda)\Gamma_0 f_{\lambda}=\Gamma_1 f_{\lambda}, \
f_{\lambda}\in \ker (A_{\max}-\lambda),\  \lambda\in
\mathbb{C}_\pm,
\end{equation*}
is called the Weyl-Titchmarsh $M$-function of the operator
$A_{\max}$ w.r.t. the corresponding boundary triple.
\end{defn}

The property of the Weyl-Titchmarsh $M$-function that makes it the
tool of choice for the analysis of isospectral Laplacians on
graphs can be formulated \cite{RyzhovOTAA} in the following way:
provided that $A_B$ is an almost solvable extension of a
simple\footnote{I.e., there exists no reducing subspace $H_0$ such
that the restriction $A_{\min}|H_0$ is a selfadjoint operator in
$H_0.$} symmetric operator $A_{\min}$, $\lambda_0\in \rho(A_B)$ if
and only if $(B-M(\lambda))^{-1}$ admits analytic continuation
into the point $\lambda_0$.

In \cite{Yorzh3}, we have obtained the following
\begin{prop}[\cite{Yorzh3}]\label{Prop_M}
Let $\Gamma_{\delta}$ be a marked compact metric graph. There
exists a closed densely defined symmetric operator $A_{\min}$ and
a boundary triple such that the operator $A_{\vec{\alpha}}$ is an
almost solvable extension of $A_{\min}$, for which the
parameterizing matrix $B$ is nothing but
$\mathrm{diag}\{\alpha_1,\dots,\alpha_N\}$, whereas the
generalized Weyl-Titchmarsh $M$-function is a $N\times N$ matrix
with matrix elements given by the following formula for a vertex
$V_k$ of $\delta$ type: $m_{jk}(\lambda)=$
\begin{equation}\label{Eq_Weyl_Func_Delta}
\begin{cases}\scriptsize
-\mu\Bigl(\sum_{e_t\in E_k}\cot\mu l_t-\sum_{e_t\in
E'_k}\tan\mu l_t-    \\
-2\sum_{e_t\in L_k}\tan\frac{\mu l_t}{2}\Bigr),
            & j=k,\\
\mu\sum_{e_t\in C_{kj}}\frac{1}{\sin\mu l_t},
            & j\not=k,\  V_j \mbox{ is a vertex of }\\
            & \delta\mbox{ type adjacent to}\ V_k,\\
-\sum_{e_t\in C'_{kj}}\frac{1}{\cos\mu l_t},
            & j\not=k,\  V_j \mbox{ is a vertex of }\\
            & \delta'\mbox{ type adjacent to }\ V_k,\\
            0,
            & j\not=k,\  V_j \mbox{ is a vertex }\\
            & \mbox{ not adjacent to }\ V_k,\\
     \end{cases}
\end{equation}
and by the following formula for a vertex $V_k$ of $\delta'$ type:
$m_{jk}(\lambda)=$
\begin{equation}\label{Eq_Weyl_Func_Delta'}
\begin{cases}
-\frac{1}{\mu}\Bigl(\sum_{e_t\in E_k}\cot\mu l_t-\sum_{e_t\in
E'_k}\tan\mu
l_t+\\
+2\sum_{e_t\in L_k}\cot\frac{\mu
 l_t}{2}\Bigr),  & j=k,\\
-\sum_{e_t\in C'_{kj}}\frac{1}{\cos\mu l_t},
            & j\not=k,\  V_j \mbox{ is a vertex of }\\
            & \delta\mbox{ type adjacent to }\ V_k,\\
-\frac{1}{\mu}\sum_{e_t\in C_{kj}}\frac{1}{\sin\mu l_t},
            & j\not=k,\  V_j  \mbox{ is a vertex of }\\
            & \delta'\mbox{ type adjacent to }\ V_k,\\
            0,
            & j\not=k,\  V_j \mbox{ is a vertex}\\
            & \mbox{not adjacent to }\ V_k.\\
     \end{cases}
\end{equation}
Here $\mu=\sqrt{\lambda}$ (the branch such that $\im\mu\geq 0$),
$l_t$ is the length of $e_t$, $L_k$ is the set of loops at the
vertex $V_k,$ $E_k$ is the set of graph edges incident to the
vertex $V_k$ with both endpoints of the same type,  $E'_k$ is the
set of graph edges incident to the vertex $V_k$ with endpoints of
different types, $C_{kj}$ is the set of graph edges connecting
vertices $V_k$ and $V_j$ of the same type, and finally, $C'_{kj}$
is the set of graph edges connecting vertices  $V_k$ and $V_j$ of
different types.
\end{prop}

This explicit description of the $M$-function w.r.t. the boundary
triple that we are inclined to consider as the natural one has
allowed us to prove necessary conditions of quantum graphs
isospectrality. In order to achieve this goal, one has to look as
the asymptotic expansion of the M-function as $\lambda\to-\infty$
along the real line, see details in \cite{Yorzh3}.

\section{Necessary and sufficient conditions of isospectrality}

In the present Section we will obtain necessary and sufficient
conditions of isospectrality provided that edge lengths of the
graph are assumed to be rationally independent.

In what follows we will utilize the following three rather
straightforward lemmata.

\begin{lem}\label{Lemma_entire}
Let $A_{\vec{\alpha}}$ be a Laplacian on the marked graph
$\Gamma_{\delta}$. If all edge lengths of the graph
$\Gamma_\delta$ are rationally independent, the function
$\Pi(\lambda)\det(M(\lambda)-B)$ with
$B=\mathrm{diag}\{\alpha_1,\dots,\alpha_N\}$ and
\begin{equation}\label{Eq_Pi}
\Pi(\lambda)=\lambda^\rho\prod_{e_j\in \mathcal C}\frac{\sin l_j
\sqrt{\lambda}}{\sqrt{\lambda}} \prod_{e_j\in \mathcal C'}{\cos
l_j \sqrt{\lambda}}\prod_{e_j\in \mathcal L}{\cos l_j
\frac{\sqrt{\lambda}}2}\prod_{e_j\in \mathcal L'}\frac{\sin l_j
\frac{\sqrt{\lambda}}2}{\sqrt{\lambda}}
\end{equation}
is an entire function of exponential type and order not greater
than $1/2$ in $\mathbb C$, if $\rho$ is the number of vertices of
$\delta'$ type connected in $\Gamma_\delta$ with at least one
vertex of this same type (either via an edge or a loop).

Here $\mathcal C$ ($\mathcal C'$) is the set of edges connecting
vertices of same type (of different types, respectively);
$\mathcal L$ ($\mathcal L'$) is the set of loops attached to
vertices of type $\delta$ (of type $\delta'$, respectively).
\end{lem}
\begin{proof}
Due to Proposition \ref{Prop_M} it suffices to prove that the
function in question has no poles in $\mathbb C$. Then the fact
that it has claimed exponential type and order follows immediately
from \cite{Levin} since it can be represented as a fraction of two
entire functions of this same order.

From Proposition \ref{Prop_M} we deduce that poles could only be
located at zeroes of the function $\Pi(\lambda)$. Consider the
possibilities.

(i) Let the edge $e_t$ of length $l_t$ connect vertices $V_k$ and
$V_j$ both of type  $\delta$. Then by \eqref{Eq_Weyl_Func_Delta}
the entries $m_{kk},\ m_{jj}$ of
 $M(\lambda)$ contain the summand $-\mu\cot\mu l_t$ corresponding to $e_t$, whereas the entries $m_{kj},\ m_{jk}$ contain
 $\mu\csc\mu l_t$. Moreover, no other entries of
$M(\lambda)$ contain trigonometric functions of the same argument.
 Add the $j$th row multiplied by $\cos\mu l_t$ to the $k$th one.
Then in $m_{kj}$ the term
 $\mu\csc\mu l_t$ admits the form  $\mu\sin\mu l_t$,
whereas the term  $-\mu\cot\mu l_t$ in $m_{kk}$ cancels out.
Factoring out $\frac{1}{\sin\mu l_t}$ from the $j$th row we see
that the function $\Pi(\lambda)\det(M(\lambda)-B)$ has no poles
associated with zeroes of ${\sin\mu l_t}$.

(ii) If the edge $e_t$ of length $l_t$ connects vertices
 $V_k$ and $V_j$ both of type $\delta'$, the same argument
 as in (i) above applies.

(iii) Let the edge $e_t$ of length $l_t$ connect vertices
 $V_k$ and $V_j$ having $\delta$ and $\delta'$ types, respectively. In this case by \eqref{Eq_Weyl_Func_Delta'}  $m_{kk}$ contains the summand $\mu\tan\mu
l_t$ and $m_{jj}$ contains $\frac{1}{\mu}\tan\mu l_t$,
corresponding to this edge, whereas the entries $m_{kj},\ m_{jk}$
contain the term $-\sec\mu l_t$. Moreover, no other matrix
elements of $M(\lambda)$ contain trigonometric functions of the
same argument.  Add the $j$th row multiplied by $\mu\cos\mu l_t$
to the $k$th one. Then in $m_{kj}$ the term
 $-\sec\mu l_t$ admits the form  $-\cos\mu l_t$,
whereas the term  $\mu\tan\mu l_t$ in $m_{kk}$ cancels out.
Factoring out $\frac{1}{\cos\mu l_t}$ from the $j$th row we see
that the function $\Pi(\lambda)\det(M(\lambda)-B)$ has no poles
associated with zeroes of ${\cos\mu l_t}$.

(iv) The case of loops attached to vertices of either $\delta$ or
$\delta'$ type is of course trivial, see
\eqref{Eq_Weyl_Func_Delta} and \eqref{Eq_Weyl_Func_Delta'}.

Finally, as it is easily seen, the factor $\lambda^\rho$ ensures
that the function in question has no pole at zero, cf. Theorem
\ref{Thm_det} below.
\end{proof}

\begin{lem}\label{Lemma_ness_and_suff}
Assume that $A_{\vec{\alpha}}$ and $A_{\vec{\widetilde{\alpha}}}$
are two Laplacians on the graph $\Gamma_{\delta}$. Assume that all
edge lengths are rationally independent. These operators are
isospectral if and only if the numbers of zero coupling constants
at vertices of $\delta'$ type are equal in $\vec\alpha$ and
$\vec{\tilde\alpha}$ and
\begin{equation}\label{eq_lemma_ness_and_suff}
\Phi(\vec\alpha)\Pi(\lambda)\det(M(\lambda)-B)\equiv
\Phi(\vec{\tilde\alpha}) \Pi(\lambda)\det(M(\lambda)-\tilde B)
\text{ for all } \lambda\in\mathbb C,
\end{equation}
where $B=\mathrm{diag}\{\alpha_1,\dots,\alpha_N\}$ and $\tilde
B=\mathrm{diag}\{\tilde\alpha_1,\dots,\tilde\alpha_N\}$;
\begin{equation}\label{eq_def_phi}
\Phi(\vec\alpha):=\prod_{V_j\text{ of }\delta'\text{ type:
}\alpha_j\not=0}\frac {\deg V_j}{\alpha_j}.
\end{equation}
\end{lem}

\begin{proof}
The ``only if'' part follows from Lemma \ref{Lemma_entire} and
Proposition \ref{Prop_M} by Hadamard theorem, see \cite{Yorzh3}
for details.

Consider the ``if'' part. In the case when $A_{\min}$ is simple
(i.e., contains no reducing self-adjoint ``parts'') the proof is
trivially based on the fact mentioned above that ``zeroes'' of
$M(\lambda)-B$ are located exactly at the points of spectrum of
the operator $A_B\equiv A_{\vec{\alpha}}$ taking into account that
$M$-function is an R-function in both half-planes of the complex
plane with almost everywhere Hermitian values on $\mathbb R$, see
\cite{DM}.

By \cite{Karpenko} the condition of rational independence of edge
lengths in fact guarantees this simplicity whenever the graph
contains no loops.

If the graph contains loops, we are sure to have parts of the
spectrum of $A_B$ ``invisible'' to the function $M(\lambda)-B$,
but due to \cite{Yorzh1,Karpenko} the ``invisible'' part of the
spectrum coincides with the point spectrum of $A_{\min}$, which is
identical for $A_B$ and $A_{\tilde B}$.
\end{proof}

We remark that the condition that the numbers of zero coupling
constants at vertices of $\delta'$ type are equal in $\vec\alpha$
and $\vec{\tilde\alpha}$ is not needed to prove the ``if'' part of
the latter Lemma since by an argument of \cite{Yorzh3} it already
follows from $\det(M(\lambda)-B)/\det(M(\lambda)-\tilde B)=const.$
This condition is nonetheless vital for correctness of
\eqref{eq_def_phi} and \eqref{eq_lemma_ness_and_suff}.

Assume that $M$ is a symmetric $N\times N$ matrix with matrix
elements $m_{ij}$. Let the weighted \emph{oriented} graph
$\Gamma(M)$ of exactly $N$ vertices be constructed in the
following way. For every $i=1,\dots,N$ attach a loop carrying the
weight $m_{ii}$ to the $i$th vertex. For all $i\not=j$ such that
the matrix element $m_{ij}$ is non-zero, draw two edges (one in
each direction) connecting the $i$th and $j$th vertices, each
carrying the weight $m_{ij}$.

The following Lemma holds.

\begin{lem}\label{Lemma_determinant}
Let $M$ be a symmetric $N\times N$ matrix with matrix elements
$m_{ij}$; let $\Gamma(M)$ be the weighted oriented graph
associated with $M$ as described above.

Let $\mathcal G$ be the set of all spanning (i.e., containing all
vertices of $\Gamma(M)$) subgraphs of $\Gamma(M)$ which are unions
of non-intersecting loops and properly oriented cycles.

Then one has the following formula for $\det M$:
$$ \mathrm{det} M = \sum_{G\in \mathcal G}
(-1)^{N(G)}\prod\limits_{e\in E(G)}w(e),
$$
where the sum is taken over all subgraphs $G\in \mathcal G$;
$E(G)$ is the set of edges of $G$, $w(e)$ is the weight of the
edge $e$ and finally $N(G)$ is the number of cycles containing
even number of edges (even cycles) in $G$.
\end{lem}
\begin{rem}
This Lemma is probably well known in some similar form. We have
picked up the idea in \cite{Tutte}, where it is used in the
framework of spectral theory of discrete Laplace operators (and
then $\sum_{j=1}^N m_{ij}=0$ $\forall i$, which simplifies the
result). For the sake of completeness, we provide the proof below.
\end{rem}

\begin{proof}
By the definition of determinant,
$$
\mathrm{det} M = \sum\limits_{\{a_1,...,a_n\}}
(-1)^{\tau(a_1,a_2,...,a_n)} m_{1a_1}...m_{na_n},
$$
where $\{a_1,...,a_n\}$ is a permutation of the set
$\{1,\dots,n\}$ and $\tau\{a_1,a_2,...,a_n\}$ is its parity.
Consider an arbitrary nonzero monomial $m_{1a_1}...m_{na_n}$ of
this sum. All factors of the form $m_{ii}$ correspond to loops of
 $\Gamma(M)$. Let exactly $k$ factors in this monomial be diagonal elements; consider the remaining part of it,
$m_{i_1a_{i_1}}...m_{i_{N-k}a_{i_{N-k}}}$. The permutation
 $\{a_{i_1},...,a_{i_{N-k}}\}$ of the set $\{i_1,...,i_{N-k}\}$ can obviously be  decomposed into the product of independent
 permutations, each of which corresponds to a certain simple cycle of $\Gamma(M)$.
Moreover, these cycles do not have common vertices.

Therefore, the monomial $m_{1a_1}...m_{na_n}$ is nothing but the
product of edge weights of a spanning subgraph $G\in\Gamma(M)$
which is a collection of loops and simple cycles. On the other
hand, it is easily seen that for any subgraph of this type there
exists a corresponding monomial in $\det M$.

As for the sign, $(-1)^{N(G)}=\tau\{a_1,a_2,...,a_n\}$ since the
parity of a permutation is exactly the number of even permutations
under the decomposition into the product of independent factors.
\end{proof}

We remark that if some  matrix elements $m_{ij}$ are decomposable
into sums, $m_{ji}=m_{ij}=\mu_{ij}+\nu_{ij}$, it is natural to
modify the construction of the graph $\Gamma(M)$ in a way such
that this graph: (i) has two loops with weights $\mu_{ii}$ and
$\nu_{ii}$ attached to the vertex $V_i$ in the case of $i=j$; (ii)
has two edges carrying weights $\mu_{ij}$ and $\nu_{ij}$,
respectively, directed from the vertex $V_i$ to the vertex $V_j$
\emph{and} two edges of same weights in the opposite direction. In
this situation, the proof just given applies verbatim.

An inspection of \eqref{Eq_Weyl_Func_Delta} and
\eqref{Eq_Weyl_Func_Delta'} shows how this idea can be utilized to
control the determinant of $M(\lambda)-B$. Note that this matrix
bears close resemblance to the adjacency matrix of the graph
$\Gamma_\delta$ in the sense that  (i) it is symmetric and (ii)
any non-diagonal matrix element of it is zero iff the
corresponding two vertices are disconnected in $\Gamma_\delta$. We
start by describing the corresponding modification procedure for
the graph $\Gamma_\delta$, which is to yield the graph
$\Gamma(M(\lambda)-B)$ in terms of Lemma \ref{Lemma_determinant}.

\subsubsection*{Modification of the graph
$\Gamma_{\delta}$}

1. The set of vertices of $\Gamma_{\delta}$ remains unchanged.

2. Every edge $e$ of length $l$ which is not a loop is doubled
(with both instances assigned opposite directions); the same
weight $w(e)$ is assigned to both instances:
\begin{equation*}
      \omega(e)=
     \begin{cases}
\mu\csc\mu l,
            & \text{if $e$ connects vertices}  \mbox{ of }\ \delta\mbox{ type};\\
\frac{1}{\mu}\csc\mu l,
            & \text{if $e$ connects vertices}  \mbox{ of }\ \delta'\mbox{ type}; \\

-\sec\mu l,
            & \text{if $e$ connects vertices}  \mbox{ of different types}.\\

     \end{cases}
\end{equation*}

3. Every edge $e$ of length $l$ connecting a vertex $V$ to any
other vertex $W$ induces a loop $\mathcal O$ attached to the
vertex $V$ carrying the weight $w(\mathcal O)$:
\begin{equation*}
      w(\mathcal O)=
     \begin{cases}
-\mu\cot\mu l,
            & \text{if } V, W \mbox{ are both of }\ \delta\mbox{ type};\\
-\frac{1}{\mu}\cot\mu l,
            & \text{if }V, W  \mbox{ are both of }\ \delta'\mbox{ type}; \\
\mu\tan\mu l,
            & \text{if }V \mbox{ is a vertex of }\ \delta \mbox{ type and } W  \mbox{ is of }\ \delta'\mbox{ type};\\
\frac{1}{\mu}\tan\mu l,
           & \text{if } V \mbox{ is a vertex of }\ \delta'\mbox{ type and } W  \mbox{ is of }\ \delta\mbox{ type}.\\
     \end{cases}
\end{equation*}

3. Each loop $e$ of length $l$ attached to a vertex $V$ to be
found in the graph $\Gamma_\delta$   is assigned the weight
$w(e)$:
\begin{equation*}
      w(e)=
     \begin{cases}
2\mu\tan\frac{\mu l}{2},
            &\text{if } V \mbox{ is of }\ \delta\mbox{ type};\\
-\frac{2}{\mu}\cot\frac{\mu l}{2},
            &\text{if }V  \mbox{ is of }\ \delta'\mbox{ type}. \\
     \end{cases}
\end{equation*}

4. Finally, attach a loop of weight $-\alpha_V$ to every vertex
$V$, where $\alpha_V$ is the coupling constant pertaining to the
vertex $V$. These loops will henceforth be treated specially; we
will refer to them as \emph{alpha-loops}.

We will denote the graph $\Gamma_\delta$ so modified by the symbol
$\Gamma_\delta^{(mod)}$.

\subsubsection*{Treatment of subgraphs for
$\Gamma_\delta^{(mod)}$}

Let as in Lemma \ref{Lemma_determinant} $\mathcal G$ be the set of
all spanning  subgraphs of $\Gamma_\delta^{(mod)}$ which are
unions of non-intersecting loops and properly oriented cycles.

The named Lemma then gives for $\det(M(\lambda)-B)$:
$$ \mathrm{det} (M(\lambda)-B) = \sum_{G\in \mathcal G}
(-1)^{N(G)}\prod\limits_{e\in E(G)}w(e).
$$

Each subgraph $G\in\mathcal G$ can be uniquely decomposed into a
graph which is a collection of alpha-loops $G_{\alpha}$ and the
remainder $G_{\bar\alpha}$:
$$
G=G_{\alpha}+G_{\bar\alpha}.
$$
When $G$ is the union of \emph{all} alpha-loops of
$\Gamma_\delta^{(mod)}$, for reasons of convenience we will write
$G_{\bar\alpha}:=\mathbb O$ -- the empty subgraph, to which we
ascribe the weight 1, $w(\mathbb O)=1$. The same convention will
be applied in the situation when $G$ contains no alpha-loops,
i.e., $G=G_{\bar\alpha}$; in this case, we will put
$G_\alpha:=\mathbb O$.

Let the set $\mathcal G_\alpha$ be the set of all $G_{\bar\alpha}$
including $\mathbb O$ as $G$ spans $\mathcal G$:
$$
\mathcal G_{\bar\alpha}=\{G_{\bar\alpha}|G\in\mathcal G\}.
$$
If one introduces the natural notation $w(G):=\prod_{e\in
E(G)}w(e)$ and takes into account that clearly $w(G)=w(G_\alpha)
w(G_{\bar\alpha})$, one immediately gets by Lemma
\ref{Lemma_determinant}
$$
\det (M(\lambda)-B) = \sum_{G\in \mathcal G_{\bar\alpha}}
(-1)^{N(G)}w(G_\alpha)w(G),
$$
where $G_\alpha$ is a collection of alpha-loops needed to build
$G$ up to a spanning subgraph of $\mathcal G$; obviously,
$G_\alpha$ is defined uniquely by $G$. Moreover, $w(G_\alpha)$ is
always a product (up to the sign) of those coupling constants
$\alpha_i$ which appear as the weights in $G_\alpha$.

Had no further reductions in the set $\mathcal G_{\bar\alpha}$
been possible, we would have been all done by now. Yet, clearly
$\mathcal G_{\bar\alpha}$ always contains exactly two graphs of
the form $G_1=G_0+G_f$ and $G_2=G_0+G_e$, where $G_f$ is a
properly directed loop of exactly two vertices, say $V$ and $W$,
connected in $\Gamma_\delta$ by an edge of length $l$, with equal
weights $w$ on both edges, whereas $G_e$ is a graph of two
disjoint loops of weights $w'=w'(w), w''=w''(w)$ attached to $V$
and $W$, respectively. Here $w'(w)$ and $w''(w)$ can be either:
(i) $-\mu \cot \mu l$, $-\mu \cot \mu l$, if both $V$ and $W$ are
of $\delta$ type, \emph{or} (ii) $-\mu^{-1}\cot\mu l$,
$-\mu^{-1}\cot\mu l$, if both $V$ and $W$ are of $\delta'$ type,
\emph{or}, finally, (iii) $\mu \tan\mu l$, $\mu^{-1}\tan\mu l$, if
$V$ and $W$ are of different types.

Consider the three named possibilities.

(i) If both $V$ and $W$ are of $\delta$ type, one gets
$w(G_f)=\mu^2/\sin^2(\mu l)$ and $w(G_e)=\mu^2 \cot^2(\mu l)$. By
the main trigonometric identity one then has:
\[
(-1)^{N(G_1)}w(G_1)+(-1)^{N(G_2)}w(G_2)=(-1)^{N(G_1)}\mu^2 w(G_0).
\]

(ii) If both $V$ and $W$ are of type $\delta'$, one has
$w(G_f)=\csc^2(\mu l)/\mu^2$ and $w(G_e)= \cot^2(\mu l)/\mu^2$ to
the same effect.

(iii) Finally, if $V$ and $W$ are of different types, one gets the
possibility of $w(G_f)=\sec^2(\mu l)$ and $w(G_e)= \tan^2(\mu l)$
leading to the cancellation of the same type.

Surely $G_0$ itself might be decomposable into either
$G_0=G_0'+G_f'$ or $G_0=G_0'+G_e'$, but if, say, the former of the
two decompositions holds, the subgraph $\tilde
G_0=G_0'+G_e'\in\mathcal G_{\bar\alpha}$. In turn, the
cancellation procedure for $\tilde G_0+G_f$ and $\tilde G_0+G_e$
as described above is applicable. Then one arrives at the
possibility of the corresponding cancellation procedure between
$G_0$ and $\tilde G_0$. The same argument applies to the second
named possibility.

Therefore, the cancellation described above can be applicable
repeatedly. Lemma \ref{Lemma_entire} easily implies that it
results in no subgraph weights which involve squares of
trigonometric functions; on the other hand, no higher powers are
possible due to the condition that edge lengths of $\Gamma_\delta$
are rationally independent.

The above argument gives ground to one further step in
modification of $\mathcal G_{\bar\alpha}$.

\begin{defn}\label{def_cancellation}
We define $\hat{\mathfrak{G}}$ as the set $\mathcal
G_{\bar\alpha}$, to each graph $G$ of which the following
procedure is applied.

I. Assume that $G\in\mathcal G_{\bar\alpha}$ contains two
vertices, connected in $\Gamma_\delta$ by an edge of length $l$,
to which the loops with weights of either $-\mu \cot \mu l$, $-\mu
\cot \mu l$, \emph{or} $\mu \tan\mu l$, $\mu^{-1}\tan\mu l$,
\emph{or} $-\mu^{-1}\cot\mu l$, $-\mu^{-1}\cot\mu l$ are attached,
respectively. In this case, eliminate this subgraph $G$
altogether.

II. Assuming that $G\in\mathcal G_{\bar\alpha}$ contains a
directed cycle of exactly two vertices with equal weights $w$ on
both edges, replace the weights $w$ by: (i) $\mu$, if both
vertices are of $\delta$ type; (ii) $1$, if they are of different
types and (iii) $1/\mu$, if both vertices are of type $\delta'$.

Finally, introduce the following equivalence relation. Put
$G_1\sim G_2$, if $w(G_1)=w(G_2)$, and define the factored set
\begin{equation}\label{Eq_sim}
\mathfrak{G}:=\hat{\mathfrak{G}}/\sim.
\end{equation}
\end{defn}

We have thus arrived at the following
\begin{thm}\label{Thm_det}
Assume that $A_{\vec{\alpha}}$ is a Laplacian on the graph
$\Gamma_{\delta}$ with coupling constants
$\{\alpha_1,\dots,\alpha_N\}$. Let $M(\lambda)$ be its
Weyl-Titchmarsh M-function provided by Proposition \ref{Prop_M},
and let $B=\mathrm{diag}\{\alpha_1,\dots,\alpha_N\}$. Then
\begin{equation}\label{Eq_det}
\det (M(\lambda)-B) =  \sum_{\gamma\in\mathfrak{G}}
f_\gamma(\vec\alpha)w(\gamma),
\end{equation}
where $\mathfrak{G}$ is defined by \eqref{Eq_sim}, $\gamma$ is
treated as an equivalence class belonging to $\mathfrak{G}$ with
the natural definition of $w(\gamma)$ as the weight of any of
subgraphs $G$ of this equivalence class and
$$
f_\gamma(\vec\alpha)=\sum_{G\in \gamma}(-1)^{N(G)}w(G_\alpha),
$$
$G_\alpha$ as above being a collection of alpha-loops needed to
build $G$ up to a spanning subgraph of $\mathcal G$, uniquely
determined by $G$.
\end{thm}

\begin{rem}
One can easily ascertain that in the situation of all vertices of
$\Gamma_\delta$ being of the same type (i.e., either $\delta$ or
$\delta'$) the function $f_\gamma(\vec\alpha)$ for each
$\gamma\in\mathfrak{G}$ is homogeneous w.r.t. its argument. It is
in particular linear in $\{\alpha_i\}_{i=1}^N$ for any class
$\gamma$ such that $G_\alpha$, $\forall G\in\gamma$ is a single
alpha-loop.
\end{rem}

The result of Theorem \ref{Thm_det} allows to reformulate Lemma
\ref{Lemma_ness_and_suff} in the following explicit form.

\begin{thm}\label{Thm_n_and_s}
Assume that $A_{\vec{\alpha}}$ and $A_{\vec{\widetilde{\alpha}}}$
are two Laplacians on the graph $\Gamma_{\delta}$. Assume that all
 edge lengths are rationally independent. These operators are
isospectral if and only if the numbers of zero coupling constants
at vertices of $\delta'$ type are equal and
\begin{equation}\label{eq_n_and_s}
\Phi(\vec\alpha) f_\gamma(\vec\alpha)=\Phi(\vec{\tilde\alpha})
f_\gamma(\vec{\tilde\alpha}), \quad \forall \gamma\in\mathfrak{G},
\end{equation}
where $f_\gamma$ and $\mathfrak{G}$ are as in Theorem
\ref{Thm_det}; the function $\Phi$ is defined in Lemma
\ref{Lemma_ness_and_suff}.
\end{thm}

\begin{proof}
The ``only if'' part follows immediately from Lemma
\ref{Lemma_ness_and_suff}, taking into account linear independence
of all functions of the form $\Pi(\lambda)w(\gamma)$, provided for
by the condition of rational independence of edge lengths.

The ``if'' part is trivial by Theorem \ref{Thm_det} and Lemma
\ref{Lemma_ness_and_suff}.
\end{proof}



One has to admit that the result of Theorem
\ref{Lemma_ness_and_suff} at the first glance does not seem to be
suitable for applications. Indeed, for a given graph
$\Gamma_\delta$, even having explicitly calculated  functions
$f_\gamma$, one faces the necessity to consider all possibilities
for $\Phi(\vec\alpha)$ and $\Phi(\vec{\tilde\alpha})$ in turn.
These are to include all admissible configurations of zero
coupling constants at $\delta'$ type vertices independently for
$\vec\alpha$ and $\vec{\tilde\alpha}$. Each of these
configurations on the face of it gives rise to a different set of
necessary and sufficient conditions of isospectrality.

However, matters do simplify to a great extent provided that one
faces the situation of $\Phi(\vec\alpha)=\Phi(\vec{\tilde\alpha})$
under the assumption of isospectrality. Indeed, not only the
awkward first factors then disappear from conditions
\eqref{eq_n_and_s}, but functions $f_\gamma$ also admit further
simplification. In order to ascertain this, note that by
construction some of functions $f_\gamma$ may contain constant
summands. These clearly come from $\gamma\in\mathfrak G$
containing at least one $G\in \mathcal G$ such that
$G_\alpha=\mathbb O$. Under the additional assumption
$\Phi(\vec\alpha)=\Phi(\vec{\tilde\alpha})$ these constant terms
can clearly be dropped altogether. For any $\gamma$ decompose
$f_\gamma(\vec\alpha)=c_\gamma+g_\gamma(\vec\alpha)$, where
$c_\gamma$ is the term in $f_\gamma$ independent of $\vec\alpha$.
Then one has

\begin{cor}\label{cor_n_and_s}
Assume that $A_{\vec{\alpha}}$ and $A_{\vec{\widetilde{\alpha}}}$
are two Laplacians on the graph $\Gamma_{\delta}$. Assume that all
edge lengths are rationally independent. Let the condition of
isospectrality yield
\begin{equation}\label{eq_phi}\Phi(\vec\alpha)=\Phi(\vec{\tilde\alpha}).\end{equation} Then the
operators are isospectral if and only if the numbers of zero
coupling constants at vertices of $\delta'$ type are equal and
\begin{equation}\label{eq_n_and_s1}
g_\gamma(\vec\alpha)= g_\gamma(\vec{\tilde\alpha}), \quad \forall
\gamma\in\mathfrak{G}.
\end{equation}
\end{cor}

It turns out that in general \eqref{eq_phi} cannot be guaranteed
(see Example \ref{A3} below) even in the situation of rationally
independent edge lengths. Nevertheless one might argue that
\emph{almost all} graphs $\Gamma_\delta$ have this property. This
claim is in fact a corollary of what is to follow.

Note first that equality \eqref{eq_phi} under the assumption of
isospectrality is in fact guaranteed \cite{Yorzh3} in the case
when all graph vertices are of the same type (trivially in the
$\delta$ case).

In the general situation, first assume that the graph
$\Gamma_\delta$ is not a tree. Then one is sure to have a subgraph
$G\in\mathcal G$ (cf. Lemma \ref{Lemma_determinant}) with
$G_\alpha=\mathbb O$ such that for no graph $\tilde G\in \mathcal
G$ with $G_\alpha\neq \mathbb O$ the equality $w(\tilde
G_{\bar\alpha})=w(G)$ holds.

This subgraph is nothing but a set of $N$ regular loops (i.e.,
non-alpha loops).

The proof of this is a straightforward application of the
following simple

\begin{lem}\label{Lemma_marking}
Assume that $\Gamma$ is a graph with each edge assigned a
different weight $l_j$. Call a marking of graph vertices
\emph{admissible} if any vertex is allowed to be marked by either
of the weights $l_j$ of the edges incident to it \emph{and} no
edge has both endpoints marked by its weight at the same time,
unless this edge is a loop. Then

(i) if the graph is not a tree, there exists an admissible marking
of vertices such that no proper subgraph $G\subsetneq\Gamma$ has
an admissible marking with the same set of marks;

(ii) if the graph is a tree, there exists no admissible marking.
In this situation, the maximal number of different weights
appearing in a marking equals $n=N-1$; this ``maximal'' marking
has exactly one edge having both endpoints marked with its weight.
For any edge of the graph, a maximal marking is constructed
uniquely such that this particular edge and this edge only has
both endpoints marked by the same weight.
\end{lem}

\begin{proof}
We start with (ii). Take any vertex $V$ of the tree $\Gamma$ as
the root. Mark any vertex $V'$ adjacent to it with the weight of
the edge connecting it to $V$. Repeat this step for every $V'$ in
the r\^ole of $V$ and carry on until the graph is over. Now pick
any of the edges incident to the root $V$ and mark $V$ with its
weight. The claim follows.

Consider (i). Pick a vertex $V$ belonging to at least one cycle of
the graph or a vertex $V$ with a loop attached to it provided that
$\Gamma$ contains no cycles. Take a spanning tree $T$ such that at
least one edge $e$ of $\Gamma$ incident to $V$ is missing from it.
Mark all vertices but $V$ as above, then mark the vertex $V$ with
the weight of $e$.

The second part of the claim follows from the fact that any proper
subgraph of $\Gamma$ has at least one vertex less than $\Gamma$.
By construction of the marking, there are no repeating weights in
it, thus any proper subgraph will have at least one weight less.
\end{proof}

Applying Theorem \ref{Thm_n_and_s} to a pair of graph Laplacians
on a non-tree graph $\Gamma_\delta$ which are assumed to be
isospectral one  then immediately arrives at the identity
$\Phi(\vec{\alpha})\equiv\Phi(\vec{\tilde\alpha})$.

We now pass over to the surprisingly much more involved analysis
of the situation when $\Gamma_\delta$ is a tree.

Assume again that we have two isospectral Laplacians on the graph
$\Gamma_\delta$. Then by \cite{Yorzh3}, one has the set equality
$S(\vec{\alpha})=S(\vec{\tilde\alpha})$, where
$S(\vec{\alpha}):=\{\sigma_1,\dots,\sigma_N\}$ with
\begin{equation}\label{eq_sigma}
\sigma_i:=\begin{cases}
 -\alpha_i/\deg V_i& \text{ if } V_i \text{ is of type } \delta\\
 \deg V_i/\alpha_i& \text{ if } V_i \text{ is of type } \delta'
 \text{ and } \alpha_i\not=0\\
 0& \text{ otherwise}
\end{cases}
\end{equation}
for any $i=1,\dots,N$.

It follows that
\begin{equation}\label{eq_products}
\prod_{\delta}{\vphantom{\prod}}^\prime (-\alpha_i/\deg
V_i)\prod_{\delta'}{\vphantom{\prod}}^\prime (\deg
V_i/\alpha_i)=\prod_{\delta}{\vphantom{\prod}}^\prime
(-\tilde\alpha_i/\deg V_i)\prod_{\delta'}{\vphantom{\prod}}^\prime
(\deg V_i/\tilde\alpha_i),
\end{equation}
 where the symbol $\prod'$ means that
the product is taken over non-zero values of $\alpha_i$; the
products are taken over all vertices of types indicated by the
subscript.

Consider the case when there are no zero coupling constants in
$\vec{\alpha}$  (and hence, in $\vec{\tilde\alpha}$). The
reduction of the general situation to the named one will be
discussed in the next Section.

The identity \eqref{eq_products} implies that the only way of
having $\Phi(\vec\alpha)\neq \Phi(\vec{\tilde\alpha})$ under the
assumption of isospectrality is to have a number of values
$\sigma_i$ redistributed somehow between graph vertices of
different types.

We will henceforth assume that every edge of the graph
$\Gamma_\delta$ is of ``mixed'' type (i.e., its endpoints have
different types) and consider three different cases separately.
The general case can be reduced to this one, see Section 4 for
details.

\subsubsection*{A vertex of valence 2} Assume w.l.o.g. that $V$ is
a $\delta$ type vertex of valence 2 (the case of $\delta'$ type
can be treated along the same lines). Consider the following
special choice of $G\in\mathcal G$: take an alpha-loop at each
graph vertex of $\delta'$ type and the alpha-loop at the vertex
$V$. Take a regular loop at each other vertex. In order to fix $G$
uniquely, we use the marking provided by Lemma \ref{Lemma_marking}
for the graph $\Gamma$ under the assumption that the vertex $V$ is
taken as a root; clearly, since at $V$ subgraph $G$ has an
alpha-loop, any other $\delta$ type vertex will then contribute a
loop of a unique weight to $G$.

Consider all subgraphs $\tilde G\in\mathcal G$ which are
equivalent to $G$ in the following sense: $\tilde
G_{\bar\alpha}\sim G_{\bar\alpha}$ under the usual decomposition
$G=G_{\alpha}+G_{\bar\alpha}$ and $\tilde G=\tilde G_\alpha+\tilde
G_{\bar\alpha}$. This set of subgraphs $\tilde G$ can be
explicitly described. Note that any edge $e$ incident to the
vertex $V$ contributes alpha-loops associated to both its
endpoints to $G$. For a fixed $e$ replace now these two
alpha-loops by the corresponding directed cycle of two vertices.
Clearly, the subgraph $\tilde G_e$ thus constructed is equivalent
to $G$ since the weight of a two-cycle is equal to 1.

There are no other subgraphs $\tilde G\in\mathcal G$ equivalent to
$G$. To prove this claim, one has to note that any subgraph
equivalent to $G$ must not only have the same collection of
tangents $\tan l_j\mu$, but also the same combined power of $\mu$
in its weight (cf. Proposition \ref{Prop_M}). Since all graph
vertices of $\delta$ type carry loops with weights of the form
$\mu\tan l_j\mu$, whereas vertices of $\delta'$ type have loops
with $\mu^{-1}\tan l_j\mu$ weights, $G$ by construction has the
factor $\mu^{\#(\delta)-1}$ in its weight, where $\#(\delta)$ is
the number of $\delta$ type vertices of $\Gamma$. It follows that
any subgraph equivalent to $G$ must have exactly the same number
of loops coming from $\delta$ type vertices, i.e., $\#(\delta)-1$.
Since the selection of edge lengths appearing in tangents is fixed
by the choice of $G$ and at the same time there are no edges of
$\delta-\delta$ and $\delta'-\delta'$ type in $\Gamma$ by
assumption, this implies: away from $V$ and the vertices adjacent
to it $\tilde G$ should be exactly the same as $G$. Trivial
combinatorics now completes the proof.

By Theorem \ref{Thm_n_and_s} one then easily obtains the following
formula linking together coupling constants of two isospectral
Laplacians:
\begin{equation}\label{eq_balance}
\alpha_V-\sum_{V' \text{ adjacent to }V} \frac
1{\alpha_{V'}}=\tilde\alpha_V-\sum_{V' \text{ adjacent to }V}
\frac 1{\tilde\alpha_{V'}},
\end{equation}
where $\alpha_V$ is the coupling constant at the vertex $V$. We
will henceforth refer to this condition as to \emph{the balancing
condition at the vertex $V$}.

One can go one step further. Consider any of the two subgraphs
$\tilde G_e$ equivalent to $G$ described above. Modify $\tilde
G_e$ in the following way. Replace the directed two-cycle
corresponding to the edge $e$ of $\Gamma$ by two regular loops in
a way such that there are no repeating weights in the subgraph
loops still. There clearly exists a unique way of doing so. Denote
the subgraph thus constructed by $\hat G_e$. The weight of $\hat
G_e$ contains exactly the same power of $\mu$ as that of $\tilde
G_e$ and $G$, differing from them in two additional tangents. By
inspection, the weights of $\hat G_e$ for \emph{both} edges $e$
incident to the vertex $V$ contain exactly the same collection of
tangents. An argument similar to the one presented above shows
that there are no further subgraphs equivalent to $\hat G_e$.
Theorem \ref{Lemma_ness_and_suff} applied to the class of
equivalence $\gamma$ such that $\hat G_e\in\gamma$ now implies
that
$$
\sum_{V' \text{ adjacent to }V} \frac 1{\alpha_{V'}}=\sum_{V'
\text{ adjacent to }V} \frac 1{\tilde\alpha_{V'}},
$$
and thus $\alpha_V=\tilde\alpha_V$ by the balancing condition
\eqref{eq_balance} for the vertex $V$ considered.

\subsubsection*{A vertex of valence 1} Now consider the case of a
boundary vertex $V$ of $\delta$ type (the case of $\delta'$ type
is treated in a similar way).

The balancing condition \eqref{eq_balance} is obtained as in the
case of $\deg V=2$ discussed above, verbatim. The difference with
the latter case is that here one cannot take the second step
described above. This follows from the fact that for a boundary
vertex there exists no possibility to construct $\hat G_e$: any
attempt at doing so results in a subgraph having two loops with
exactly the same weights at adjacent vertices. On the other hand,
such subgraphs have to be eliminated by Definition
\ref{def_cancellation} and thus do not belong to $\mathfrak{G}$.

\subsubsection*{A vertex of valence greater than 2} Finally assume
that $V$ is a $\delta$ type (w.l.o.g.: the case of $\delta'$ is
treated similarly) vertex of higher valence.

Proceeding as above, one gets the balancing condition
\eqref{eq_balance}.

As in the case of $\deg V=2$, one can go one step further. In
order to do so, note that in the case of higher valences there are
$C_d^2$ non-directed 2-paths through the vertex $V$, where $d=\deg
V$. Each of these paths gives rise to a pair of \emph{equivalent}
subgraphs $\hat G_e$, $\hat G_{e'}$, where $(e,e')$ is the
corresponding path through $V$. Moreover, it is easily seen that
for different paths $e,e'$ one arrives at different equivalence
classes $\gamma\in\mathfrak G$. A straightforward argument shows
that by Theorem \ref{Thm_n_and_s} one finally arrives at the
following system of $C_d^2$ linear equations on $d$ variables:
$$
x_e+x_{e'}=\tilde x_e+\tilde x_{e'} \text{ for any path } (e,e')
\text{ through } V,
$$
the new variables being introduced as follows:
$$
x_e:=\frac 1{\alpha_{V_e}};\quad \tilde x_e:=\frac 1{
\tilde\alpha_{V_e}},
$$
where $V_e$ is a vertex such that together with $V$ it constitutes
the pair of endpoints of the edge $e$, incident to $V$. Clearly
this system admits the only solution $x_e=\tilde x_e$ for any $e$
incident to the vertex $V$, provided that $d>2$.

Therefore, in the case of valences higher than 2 one not only gets
$\alpha_V=\tilde\alpha_V$ for the vertex $V$ itself as in the case
of $\deg V=2$, but also automatically
$\alpha_{V'}=\tilde\alpha_{V'}$ for any vertex $V'$ adjacent to
$V$.

The argument presented above in fact applies to an arbitrary graph
$\Gamma_\delta$ rather than to just a tree. One has the following

\begin{thm}\label{thm_uniqueness_nonzero}
Assume that $A_{\vec{\alpha}}$ and $A_{\vec{\widetilde{\alpha}}}$
are two Laplacians on an arbitrary compact metric graph
$\Gamma_{\delta}$ with all edges (except loops) of mixed type.
Suppose that the number of vertices is greater than
two\footnote{Clearly the degenerate case of $N=2$ leads to trivial
isospectrality since one can always swap both coupling constants
which leads to the same graph Laplacian.}. Assume that all edge
lengths are rationally independent. Let $A_{\vec{\alpha}}$ and
$A_{\vec{\widetilde{\alpha}}}$ be isospectral for $\vec\alpha$ and
$\vec{\tilde\alpha}$  such that  $\alpha_i\neq 0$ (and hence
$\tilde\alpha_i\neq 0$) for all $i=1,\dots,N$. Then
$\alpha_i=\tilde \alpha_i$ for all $i=1,\dots,N$.
\end{thm}

\begin{proof}
The case of a tree graph has in fact been proven above. Indeed,
one immediately obtains the claimed result for all internal
vertices. Then one uses balancing conditions at all boundary
vertices.

Consider the situation when $\Gamma_\delta$ has no multiple edges
and no loops. Then for any vertex $V$ of the graph there exists a
spanning tree $T$ such that every edge incident to $V$ belongs to
it. One can then proceed analogously to the case of a tree graph,
cf. the proof of Lemma \ref{Lemma_marking}.

The situation when $\Gamma_\delta$ is allowed to contain multiple
edges is somewhat different, although can be considered in much
the same way. The balancing condition \eqref{eq_balance} in
particular admits the form
\begin{equation}\label{eq_balance_general}
\alpha_V-\sum_{V' \text{ adjacent to }V} \frac
{\nu(V')}{\alpha_{V'}}=\tilde\alpha_V-\sum_{V' \text{ adjacent to
}V} \frac {\nu(V')}{\tilde\alpha_{V'}},
\end{equation}
where $\nu(V')$ is the multiplicity of edge connection betwen $V$
and $V'$, i.e., the number of edges having vertices $V$ and $V'$
as their endpoints. The rest of the proof remains virtually
unchanged.

Finally, one notes that the possible presence of loops does not
change the argument a single bit; balancing conditions at graph
vertices do not involve any loops-related information.
\end{proof}

\begin{rem}
Balancing conditions \eqref{eq_balance_general} at the first
glance look awkwardly non-linear. Nevertheless, passing over to
the variables $\sigma_j$ defined in \eqref{eq_sigma} one arrives
at the following systems of balancing equations:
$$
\deg V\sigma_V+\sum_{V'\text{ adjacent to
}V}\frac{{\nu(V')}\sigma_{V'}}{\deg V'}=\deg
V\tilde\sigma_V+\sum_{V'\text{ adjacent to
}V}\frac{{\nu(V')}\tilde\sigma_{V'}}{\deg V'}
$$
for all vertices $V$ of $\delta$ type and
$$
\frac{\deg V}{\sigma_V}+\sum_{V'\text{ adjacent to
}V}\frac{\nu(V')}{\deg V'\sigma_{V'}}=\frac{\deg
V}{\tilde\sigma_V}+\sum_{V'\text{ adjacent to
}V}\frac{\nu(V')}{\deg V' \tilde\sigma_{V'}}
$$
for all vertices $V$ of $\delta'$ type, which are linear with
respect to $\{\sigma_i\}$ and $\{1/\sigma_i\}$, respectively.
\end{rem}

We now sum up all the information obtained so far pertaining to
the question of whether or not isospectrality yields
\eqref{eq_phi} in the form of the following

\begin{cor}\label{cor_phi}
Assume that $A_{\vec{\alpha}}$ and $A_{\vec{\widetilde{\alpha}}}$
are two isospectral Laplacians on the graph $\Gamma_{\delta}$.
Then if $\Gamma_\delta$ is either

(i)  a graph with all vertices of $\delta$ type or a graph with
all vertices of $\delta'$ type,

or (ii)  a non-tree graph with vertices of different types,

or (iii)  a tree with all edges of mixed type and the additional
condition $\alpha_i\neq 0$ (and hence $\tilde\alpha_i\neq0$) for
all $i=1,\dots,N$,

condition \eqref{eq_phi} holds.
\end{cor}

Tightness of this result is demonstrated by the following

\begin{eg}\label{A3}
Let $\Gamma_\delta=A_3$, i.e., the chain graph of three vertices,
with rationally independent edge lengths. Let the types of its
vertices be either $\delta-\delta'-\delta$ or
$\delta'-\delta-\delta'$. In both cases there exist non-trivial
isospectral configurations of coupling constants, leading to
$\Phi(\vec\alpha)=-\Phi(\vec{\tilde\alpha})$.
\end{eg}

\begin{proof}
A direct application of Theorem \ref{Thm_n_and_s} gives the
following description of all isospectral configurations for the
case $\delta-\delta'-\delta$:
$$
\vec\alpha=(a,\frac 2 a,0);\quad \vec{\tilde\alpha}=(0,-\frac 2
a,-a)\text { for all } a\in\mathbb R.
$$
The second case is analogous.
\end{proof}

\section{Graph reductions}

The aim of the present Section is to demonstrate how one can
reduce the problem of isospectrality for graph Laplacians defined
on $\Gamma_\delta$ to the consideration of ``smaller'', or trimmed
graphs. In particular, we will show how the general situation of
mixed $\delta$ and $\delta'$ type vertices can be reduced to the
case considered towards the end of the Section 3, i.e., to the
case when all graph edges are of mixed type.

The reductions we have in mind are actually threefold: (i)
trimming away a boundary vertex together with the edge it belongs
to; (ii) trimming away any internal edge, be it simple or
multiple; (iii) trimming away a loop attached to any graph vertex.

In all these three cases the analysis is in fact very similar. It
is based on explicit comparison of residues of
$\det(M(\lambda)-B)$ and $\det(M(\lambda)-\tilde B)$ using results
of Section 3 and in particular Theorem \ref{Thm_n_and_s}.

We start with the following setup. Let $\Gamma_\delta$ be an
arbitrary compact metric marked graph. Assume that two Laplacians
$A_{\vec\alpha}$ and $A_{\vec{\tilde\alpha}}$ are isospectral on
this graph. Further assume that the graph contains an edge $e$
with both endpoints of $\delta$ type. Consider the graph
$\Gamma_\delta^{(e)}=\Gamma_\delta-e$, i.e., the graph
$\Gamma_\delta$ with the edge $e$ removed. This means, that the
vertices $V_j$ and $V_k$ connected in the original graph by the
edge $e$ have been ``glued'' together to form the vertex
$V_{j,k}^{(e)}$ (of the same type $\delta$), whereas the rest of
the graph remains unchanged. We do allow the situation of multiple
edges, when $V_j$ and $V_k$ are connected in $\Gamma_\delta$ by
more than one edge. In this situation all these edges, but the
edge $e$, become nothing but loops of their respective lengths,
attached to the vertex $V_{j,k}^{(e)}$. We are going to argue that
under the assumption of isospectrality of $A_{\vec\alpha}$ and
$A_{\vec{\tilde\alpha}}$ one is guaranteed isospectrality of
$A_{\vec{\alpha_e}}^{(e)}$ and $A_{\vec{\tilde{\alpha_e}}}^{(e)}$,
where these two operators are graph Laplacians defined on
$\Gamma_\delta^{(e)}$ with vectors of coupling constants
$\vec\alpha_e$ and $\vec{\tilde{\alpha_e}}$, respectively. These
vectors in turn are from $\mathbb R^{N-1}$ and are constructed
based on $\vec\alpha$ and $\vec{\tilde\alpha}$ by the following
rule: a coupling constant at any unaffected vertex remains the
same, whereas the coupling constants at $V_{j,k}^{(e)}$ are
$\alpha_j+\alpha_k$ and $\tilde\alpha_j+\tilde\alpha_k$,
respectively.

The argument goes as follows. Having assumed isospectrality, one
can use the result of Lemma \ref{Lemma_ness_and_suff}. Having
divided both sides of \eqref{eq_lemma_ness_and_suff} by
$\Pi(\lambda)$ away from  poles of determinants, one can now
compute residues of both sides at the sequence of first order (by
the condition of rational independence of edge lengths) poles
$\mu_e(m)$, $m=1,2,\dots$, associated with the length $l$ of the
edge $e$. These poles are obviously located at zeroes of $\sin
l\mu$. The named residues then must coincide as well.

For any of the two determinants, one adds the $k$-th row of the
matrix $M(\lambda)-B$ multiplied by $\cos l\mu$ to the $j$-th one.
By Proposition \ref{Prop_M}, this cancels out singularities of all
matrix elements of the $j$-th row at the sequence $\mu_e(m)$. Then
the residues are calculated by passing to the limit as
$\mu\to\mu_e(m)$ in
$(\mu+\mu_e(m))(\mu-\mu_e(m))\det(M(\lambda)-B)$, where the factor
$(\mu-\mu_e(m))$ is introduced into the $k$-th row of the matrix
since by inspection all other rows are now regular at the sequence
of points $\mu_e(m)$. Thereafter, we add the $k$-th column to the
$j$-th one and, having noted that the $k$-th row now only contains
a single non-zero element, reduce the determinant to the one of a
matrix of lower dimension.

The procedure outlined above yields:
\begin{equation}\label{eq_residues}
\Phi(\vec\alpha)\det(M^{(e)}(\mu_e^2(m))-B_e)=
\Phi(\vec{\tilde\alpha})\det(M^{(e)}(\mu_e^2(m))-\tilde B_e)\quad
\forall m,
\end{equation}
where $M^{(e)}(\lambda)$ is the Weyl-Titchmarsh $M$-matrix of the
graph $\Gamma_\delta^{(e)}$ defined above, $B_e$ and $\tilde B_e$
are two diagonal matrices of coupling constants $\vec\alpha_e$ and
$\vec{\tilde{\alpha_e}}$, respectively.

Note that $\Phi(\vec\alpha)=\Phi_e(\vec\alpha_e)$ and
$\Phi(\vec{\tilde\alpha})=\Phi_e(\vec{\tilde{\alpha_e}})$ since
the set of $\delta'$ vertices and their respective coupling
constants are the same for $\Gamma_\delta$ and
$\Gamma_\delta^{(e)}$ by construction. Here the function $\Phi_e$
is defined by the same expression \eqref{eq_def_phi} as the
function $\Phi$, but for the modified graph $\Gamma_\delta^{(e)}$.

By Lemma \ref{Lemma_ness_and_suff} again, it remains to be seen
that the equality \eqref{eq_residues} holds everywhere in $\mathbb
C$ with the exception of the countable set of poles rather than
just at the sequence of points $\mu_e(m)$. In order to ascertain
this, we first multiply both sides of \eqref{eq_residues} by
$\Pi_e(\mu_e^2(m))$, where the function $\Pi_e$ is defined by the
same expression \eqref{Eq_Pi} as the function $\Pi$, but for the
graph $\Gamma_\delta^{(e)}$. By Theorem \ref{Thm_det}, one
immediately gets the following expansion in powers of $m$:
\begin{multline*}
\Phi_e(\vec{\alpha_e})\Pi_e(\mu_e^2(m))\det(M^{(e)}(\mu_e^2(m))-B_e)\\
=\sum_{j}m^j
\Psi_j(\vec{\alpha_e};l_1\mu_e(m),\dots,l_{n-1}\mu_e(m))
\end{multline*}
with functions $\Psi_j$ which are linear combinations (with
coefficients depending on $\vec{\alpha_e}$) of various products of
sines and cosines of \emph{different} arguments, belonging to the
set $\{l_1\mu_e(m), \dots, l_{n-1}\mu_e(m)\}$. Here
$l_1,\dots,l_{n-1}$ are (with a slight abuse of notation) the
collection of edge lengths pertaining to the graph
$\Gamma_\delta^{(e)}$ and are thus rationally independent.

Let $j_0$ be the topmost power of $m$ in the above representation
and consider the function
$\Psi_{j_0}(\vec{\alpha_e};l_1\mu_e(m),\dots,l_{n-1}\mu_e(m))$. On
the one hand, an application of \eqref{eq_residues} immediately
yields
\begin{multline*}
\lim_{m\to\infty}\Psi_{j_0}(\vec{\alpha_e};l_1\mu_e(m),\dots,l_{n-1}\mu_e(m))\\=\lim_{m\to\infty}\Psi_{j_0}(\vec{\tilde{\alpha_e}};l_1\mu_e(m),\dots,l_{n-1}\mu_e(m))
\end{multline*}
On the other hand, consider an analytic \cite{Shabat} function
$\Theta_{j_0}$ of $n-1$ complex variables parametrically depending
on $\vec{\alpha_e}$ which is defined in the following way:
\begin{equation}\label{eq_def_theta}
\Theta_{j_0}
(z_1,\dots,z_{n-1};\vec{\alpha_e}):=\Psi_{j_0}(\vec{\alpha_e};z_1,\dots,z_{n-1}).
\end{equation}
Note that this function is in fact $2\pi$-periodic with respect to
any of its variables restrained to the real line.

We will now use a result of \cite{Shulman} which says that for
$n-1$ rationally independent real values $\hat l_1,\dots,\hat
l_{n-1}$ (here we have put $\hat l_j:=l_j/2l$, where $l$ is the
length of the edge $e$) the set of points $(\{m\hat l_1\}, \dots,
\{m \hat l_{n-1} \})$, where $\{x\}$ denotes the fractional part
of $x$ and $m$ is a natural number, is dense in the unit cube
$\mathbb K_1\subset \mathbb R^{n-1}$.

Since $\mu_e(m)=\pi m/l$, one then has using Hartogs' theorem the
following identities for any real $z_1,\dots,z_n$ in $2\pi \mathbb
K_{1}\subset \mathbb R^{n-1}$ and a subsequence
$\{m_k\}_{k=1}^\infty$:
\begin{multline*}
\Theta_{j_0}(z_1,\dots,z_{n-1};\vec{\alpha_e})=\lim_{k\to\infty}
\Theta_{j_0}(2\pi \{\hat l_1 m_k\},\dots,\{\hat l_{n-1} m_k\};\vec{\alpha_e})=\\
\lim_{k\to\infty} \Theta_{j_0}(l_1 \mu_e(m_k),\dots,l_{n-1}
\mu_e(m_k);\vec{\alpha_e})=\\
 \lim_{k\to\infty}
\Psi_{j_0}(\vec{\alpha_e};l_1 \mu_e(m_k),\dots,l_{n-1}
\mu_e(m_k))=\\
\lim_{k\to\infty} \Psi_{j_0}(\vec{\tilde{\alpha_e}};l_1
\mu_e(m_k),\dots,l_{n-1} \mu_e(m_k))=
\Theta_{j_0}(z_1,\dots,z_{n-1};\vec{\tilde{\alpha_e}}),
\end{multline*}
where the last equality follows by reversing the first three.

Thus,
$$
\Theta_{j_0}(z_1,\dots,z_{n-1};\vec{\alpha_e})=
\Theta_{j_0}(z_1,\dots,z_{n-1};\vec{\tilde{\alpha_e}})
$$
everywhere in $2\pi\mathbb K_1$, from where by the corresponding
uniqueness theorem \cite{Shabat} it follows immediately, that the
last equality holds everywhere in $\mathbb C^{n-1}$. In
particular, it holds on the ray $z_1=l_1\mu$, $z_2=l_2 z_1/l_1$,
$\dots$, $z_{n-1}=l_{n-1}z_1/l_1$, which finally yields
$$
\Psi_{j_0}(\vec{\alpha_e};l_1\mu,\dots,l_{n-1}\mu)=\Psi_{j_0}(\vec{\tilde{\alpha_e}};l_1\mu,\dots,l_{n-1}\mu)\quad
\forall \mu\in\mathbb C.
$$
Now repeating the argument first for $\Psi_{j_1}$, where $j_1$ is
the second highest power of $\mu$, and then for all the other
consecutive powers, one obtains that
$$
\Psi_{j}(\vec{\alpha_e};l_1\mu,\dots,l_{n-1}\mu)=\Psi_{j}(\vec{\tilde{\alpha_e}};l_1\mu,\dots,l_{n-1}\mu)\quad
\forall \mu\in\mathbb C \text{ and }\forall j.
$$

Ultimately, writing down the corresponding expansion in powers of
$\mu$ for
$\Phi_e(\vec{\alpha_e})\Pi_e(\mu^2)\det(M^{(e)}(\mu^2)-B_e)$ and
comparing it with the one for
$\Phi_e(\vec{\tilde{\alpha_e}})\Pi_e(\mu^2)\det(M^{(e)}(\mu^2)-\tilde
B_e)$, one ascertains the identity sought:
\begin{equation}\label{eq_residues_finish}
\Phi_e(\vec\alpha_e)\Pi_e(\lambda)\det(M^{(e)}(\lambda)-B_e)=
\Phi_e(\vec{\tilde{\alpha_e}})\Pi_e(\lambda)\det(M^{(e)}(\lambda)-\tilde
B_e),
\end{equation}
for all $\lambda$, which completes the proof of the claim.

The corresponding result in the case of $e$ being an edge with
both endpoints of $\delta'$ type follows from a virtually
unchanged argument. The only bit that probably deserves a comment
is the following one. The identities
$\Phi(\vec\alpha)=\Phi_e(\vec\alpha_e)$ and
$\Phi(\vec{\tilde\alpha})=\Phi_e(\vec{\tilde{\alpha_e}})$ in this
setup follow from \eqref{eq_products}. This identity immediately
implies that the statement of Lemma \ref{Lemma_ness_and_suff}
holds if one replaces the factor $\Phi$ on both sides by the
function $\hat \Phi$, where
\begin{equation*}
\hat\Phi(\vec\alpha):=\prod_{V_j\text{ of }\delta\text{ type:
}\alpha_j\not=0}\frac {\deg V_j}{\alpha_j}.
\end{equation*}
Then the identities in question follow immediately since now the
set of $\delta$ type vertices and their respective coupling
constants are the same for $\Gamma_\delta$ and
$\Gamma_\delta^{(e)}$ by construction.

We now shift our attention to the case when $e$ is a loop attached
to the vertex $V$. The analysis of this situation is if anything
simpler than the one presented above. In this setup, one has to
consider the modified graph $\Gamma_\delta^{(e)}=\Gamma_\delta-V$,
that is, the graph $\Gamma_\delta$ with the vertex $V$ removed
together with all graph edges incident to it (including the loop
$e$, of course). The modified vector of coupling constants
$\vec{\alpha_e}$ is nothing but the vector $\vec{\alpha}$ with the
element pertaining to the vertex $V$ dropped (the same
modification yields $\vec{\tilde{\alpha_e}}$, of course). In order
to ensure that $\Gamma_\delta^{(e)}$ is connected, we additionally
require that the vertex $V$ in question either belongs to the
graph boundary or to a cycle of $\Gamma_\delta$.

Assume w.l.o.g. that $V$ is of $\delta$ type. The residue
calculation then trivially yields:
$$
\Phi_e(\vec\alpha_e)\det(M^{(V)}(\mu_e^2(m))-B_e)=
\Phi_e(\vec{\tilde{\alpha_e}})\det(M^{(V)}(\mu_e^2(m))-\tilde
B_e)\quad \forall m,
$$
where as above $B_e$ and $\tilde B_e$ are two diagonal matrices of
coupling constants $\vec\alpha_e$ and $\vec{\tilde{\alpha_e}}$,
respectively. The matrix $M^{(V)}$ however is not a
Weyl-Titchmarsh matrix of any graph since it is actually equal to
$M$ in which the row and the column pertaining to the vertex $V$
have been dropped.

It is nevertheless possible by the same line of argumentation as
above to ascertain the identity
$$
\Phi_e(\vec\alpha_e)\Pi_e(\lambda)\det(M^{(V)}(\lambda)-B_e)=
\Phi_e(\vec{\tilde{\alpha_e}})\Pi_e(\lambda)\det(M^{(V)}(\lambda)-\tilde
B_e)
$$
for all $\lambda\in\mathbb C$, where $\Pi_e$ has the same meaning
as in the analysis presented above. Then one is able to expand the
determinant as
$$
\det(M^{(V)}(\lambda)-B_e)=\det(M^{(e)}(\lambda)-B_e)+\Omega(\lambda),
$$
where $M^{(e)}(\lambda)$ is the Weyl-Titchmarsh matrix of the
graph $\Gamma_\delta^{(e)}$, whence by the linear independence
argument (using the fact that $\Omega(\lambda)$ is sure to contain
trigonometric functions with arguments not to be found in the
expression for $M^{(e)}$) one ultimately has the identity sought
$$
\Phi_e(\vec\alpha_e)\Pi_e(\lambda)\det(M^{(e)}(\lambda)-B_e)=
\Phi_e(\vec{\tilde{\alpha_e}})\Pi_e(\lambda)\det(M^{(e)}(\lambda)-\tilde
B_e)
$$
$\forall \lambda\in\mathbb C.$ In order to give precise
formulation of results obtained so far, we start with the
following

\begin{defn}\label{def_trimming}
Either of the following graph operations on $\Gamma_\delta$ will
be called an \emph{admissible trimming} of the graph:

(i) a removal of an edge connecting two different vertices $V$ and
$V'$ of the same type, $\Gamma_\delta^{(e)}=\Gamma_\delta-e$;

(ii) a removal of a vertex $V$ to which a loop $e$ is attached
provided that this vertex either belongs to the graph boundary or
to one of the graph cycles, $\Gamma_\delta^{(e)}=\Gamma_\delta-V$.
\end{defn}

Then the following Theorem holds.

\begin{thm}\label{thm_trimming}
Assume that $A_{\vec{\alpha}}$ and $A_{\vec{\widetilde{\alpha}}}$
are two Laplacians on the graph $\Gamma_{\delta}$. Assume that all
the edge lengths are rationally independent. Finally, let
$A_{\vec{\alpha}}$ and $A_{\vec{\widetilde{\alpha}}}$ be
isospectral. Then for any admissible trimming
$\Gamma_\delta^{(e)}$ of the graph $\Gamma_\delta$, the Laplacians
$A_{\vec{\alpha_e}}$ and $A_{\vec{\widetilde{\alpha_e}}}$ on the
trimmed graph are isospectral.

Here the vectors of coupling constants $\vec{\alpha_e}$ and
$\vec{\tilde{\alpha_e}}$ are: in the case of an admissible
trimming (i) equal to the coupling constants of $A_{\vec{\alpha}}$
and $A_{\vec{\widetilde{\alpha}}}$, respectively, for all vertices
except $V$ and $V'$ with the remaining ones being equal to
$\alpha_V+\alpha_{V'}$ and $\tilde\alpha_V+\tilde\alpha_{V'}$,
respectively; in the case of an admissible trimming (ii) equal to
the coupling constants of $A_{\vec{\alpha}}$ and
$A_{\vec{\widetilde{\alpha}}}$, respectively, for all vertices
except $V$.
\end{thm}

\begin{rem}
1. Note that a procedure of trimming away graph boundary vertices
was suggested in a similar context in \cite{Aharonov}. The latter
Theorem provides a much more general recipe of this procedure,
allowing in particular to deal with graph cycles.

2. Theorem \ref{thm_trimming} in the case of mixed $\delta$ and
$\delta'$ type vertices clearly allows to reduce the analysis of
isospectral Laplacians to the setup studied towards the end of
Section 3, i.e., when all graph edges are assumed to be of mixed
type.

3. Trimming away an edge when one of its endpoints $V$ has
$\alpha_V=\tilde\alpha_V=0$ yields more information, see Theorem
\ref{Thm_uniqueness_zero} below.
\end{rem}

The following immediate Corollary of Theorem \ref{thm_trimming}
demonstrates how one could use this result in order to prove
uniqueness (i.e., the absence of isospectral configurations of
coupling constants).

\begin{cor}\label{cor_trimming}
Assume that there exist either two different admissible trimmings
of the same type or a sequence of two consecutive admissible
trimmings of type (i) for the graph $\Gamma_\delta$. Let each of
these two trimmings result in a graph $\Gamma_\delta^{(e)}$ for
which the absence of isospectral configuration of coupling
constants is known. Then there are no isospectral configurations
for the graph $\Gamma_\delta$.
\end{cor}

Results obtained in \cite[Section 5]{Yorzh3} now yield
\begin{cor}
If $\Gamma_\delta$ is a tree graph with all vertices of the same
type (i.e., either of type $\delta$ or of type $\delta'$), the
spectrum of a graph Laplacian defined on $\Gamma_\delta$ uniquely
determines all coupling constants.
\end{cor}

A further corollary of Theorem \ref{thm_trimming} shows that in
the process of graph trimming any cycles, multiple edges and loops
(note that graph cycles reduce to multiple edges, which in turn
reduce to loops under consecutive trimmings) are in fact helpful
if one wishes to exclude the potentially problematic situation of
Example \ref{A3}, whereby some coupling constants are allowed to
zero out.

\begin{cor}
Assume that $A_{\vec{\alpha}}$ and $A_{\vec{\widetilde{\alpha}}}$
are two Laplacians on the graph $\Gamma_{\delta}$. Assume that all
the edge lengths are rationally independent.  Let the graph
contain a vertex $V$ such that: (i) $V$ has a loop attached to it
and (ii) $\alpha_V=0$, $\tilde\alpha_V\neq 0$. Then
$A_{\vec{\alpha}}$ and $A_{\vec{\widetilde{\alpha}}}$ cannot be
isospectral.
\end{cor}

\begin{proof}
Assume the opposite. Then Theorem \ref{thm_trimming} provides
isospectrality for the corresponding graph Laplacians on
$\Gamma_\delta-V$. One immediately arrives at a contradiction,
since isospectrality requires that the total number of zero
coupling constants has to be the same.
\end{proof}

The following Example demonstrates how Theorem \ref{thm_trimming}
can be used to single out isospectral configurations provided that
these in fact exist.

\begin{eg}
Consider the graph $\Gamma_\delta$ which is a pure cycle of 4
vertices with rationally independent edge lengths. Let all
vertices be of type $\delta$.

On the one hand, it is known \cite{Yorzh3} that in this situation
there are isospectral configurations of coupling constants,
namely, for an arbitrary non-zero $a\in\mathbb R$,
$\vec\alpha=(a,-a,a,-a)$, $\vec{\tilde\alpha}=(-a,a,-a,a)$, and
for these configurations only the corresponding graph Laplacians
are isospectral.

On the other hand, Theorem \ref{thm_trimming} and Corollary
\ref{cor_trimming} immediately imply, that either there are no
isospectral configurations by the result \cite{Yorzh3} pertaining
to the case of odd cycles, or
$\alpha_1+\alpha_2=0=\tilde\alpha_1+\tilde\alpha_2$ (in which case
the graph $\Gamma_\delta^{(e)}$ reduces by the procedure of
cleaning, see \cite{Aharonov}, to a cycle of two vertices for
which one might easily swap the coupling constants with no effect
on the spectrum as explained above). The remaining three
admissible trimmings of $\Gamma_\delta$ leave
$\alpha_2+\alpha_3=0=\tilde\alpha_2+\tilde\alpha_3$,
$\alpha_3+\alpha_4=0=\tilde\alpha_3+\tilde\alpha_4$ and finally
$\alpha_1+\alpha_4=0=\tilde\alpha_1+\tilde\alpha_4$ for possible
isospectral configurations. These four relations clearly yield
precisely the result of \cite{Yorzh3}.

\end{eg}

The remaining part of the paper is devoted to the analysis of
isospectrality for graph Laplacians in the situation when all
graph edges of $\Gamma_\delta$ are of mixed type, i.e., when for
each edge one of the endpoints is of type $\delta$ whereas the
other one is of $\delta'$ type. The general situation can be
reduced to the one outlined by  consecutive applications of
Theorem \ref{thm_trimming}. As no further trimming is possible
once we have arrived at a graph of this class, one has no choice
but to resort to Theorem \ref{Thm_n_and_s}. Nevertheless, there is
room still for simplifications based on residue calculus.

We are going to argue that in this case there are no isospectral
graph Laplacians unless the underlying graph is \emph{essentially}
the graph of Example \ref{A3}, i.e., the chain graph of exactly 3
vertices.

We use the following notation throughout for the sake of brevity.
Having fixed two configurations of coupling constants,
$\vec\alpha$ and $\vec{\tilde\alpha}$, which by assumption lead to
isospectral graph Laplacians, we will say that a graph vertex $V$
is of $0\bar0$ class iff $\alpha_V=0$ and $\tilde\alpha_V\neq 0$;
is of $\bar00$ class iff $\alpha_V\neq0$ and $\tilde\alpha_V= 0$;
is of $00$ class iff $\alpha_V=0$ and $\tilde\alpha_V= 0$ and
finally is of $\bar0\bar0$ class iff $\alpha_V\neq 0$ and
$\tilde\alpha_V\neq 0$. In these terms, Theorem
\ref{thm_uniqueness_nonzero} reads: if all graph vertices are of
$\bar0\bar0$ class, one has $\vec\alpha=\vec{\tilde\alpha}$.

As demonstrated by Example \ref{A3}, the situation when coupling
constants at graph vertices are allowed to zero out can in fact
lead to complications, including the possibility that the
condition \eqref{eq_phi} is not satisfied. As shown by the named
example, this in fact is due to the presence of vertices of
$0\bar0$ and $\bar00$ classes.

We start however with consideration of $00$ class of vertices. The
residue-based analysis used above will allow us to reduce the
general situation to the one when one is sure to have no vertices
of this class. The problem we face is that residue calculation in
fact removes an edge rather than a vertex as such; on the other
hand, in the situation considered endpoints of any given edge are
of different types. It is not feasible to expect that a removal of
this edge will as above leave us with an $M$-function of
\emph{any} $\Gamma_\delta$. In order to give a description of what
is to happen here, introduce the following notion.

\begin{defn}\label{def_quasigraph}
(i) \emph{Removal of a mixed edge giving rise to a vertex of
$\delta'\to\delta$ type}. The graph $\Gamma_\delta$ is subject to
the following modification resulting in the \emph{quasigraph}
$\hat \Gamma_\delta^{(e)}$. An edge $e$ connecting vertices $V$
and $V'$ of types $\delta$ and $\delta'$, respectively, is
removed. The vertices $V$ and $V'$ are glued together to form a
\emph{quasivertex} $\hat V$ of $\delta'\to\delta$ type. Any other
edge $e'$ of length $l'$ connecting $V$ and $V'$ in
$\Gamma_\delta$ gives rise to a \emph{quasiloop} of the same
length $l'$ attached to $\hat V$. Any loop $\mathcal O$ attached
to $V'$ becomes a \emph{quasiloop} of the same length attached to
$\hat V$. Any edge $e'$ incident to $V'$ and $V''\neq V$ in
$\Gamma_\delta$ becomes a \emph{quasiedge} of the same length
connecting $\hat V$ and $V''$. The rest of $\Gamma_\delta$ remains
unchanged.

(ii) \emph{Removal of a mixed edge giving rise to a vertex of
$\delta\to\delta'$ type}. The graph $\Gamma_\delta$ is subject to
the following modification resulting in the \emph{quasigraph}
$\hat \Gamma_\delta^{(e)}$. An edge $e$ connecting vertices $V$
and $V'$ of types $\delta$ and $\delta'$, respectively, is
removed. The vertices $V$ and $V'$ are glued together to form a
\emph{quasivertex} $\hat V$ of $\delta\to\delta'$ type. Any other
edge $e'$ of length $l'$ connecting $V$ and $V'$ in
$\Gamma_\delta$ gives rise to a \emph{quasiloop} of the same
length $l'$ attached to $\hat V$. Any loop $\mathcal O$ attached
to $V$ becomes a \emph{quasiloop} of the same length attached to
$\hat V$. Any edge $e'$ incident to $V$ and $V''\neq V'$ in
$\Gamma_\delta$ becomes a \emph{quasiedge} of the same length
connecting $\hat V$ and $V''$. The rest of $\Gamma_\delta$ remains
unchanged.
\end{defn}

We point out that we do not allow the procedure outlined above to
be applied to more than a single edge $e$ of a mixed type. In
other words, we do not apply Definition \ref{def_quasigraph}
consecutively.

Note that the difference between the procedures (i) and (ii)
defined above is essentially the following: one either eliminates
the vertex $V'$ by moving it to the vertex $V$ along the edge $e$,
which forces all other edges incident to $V'$ in $\Gamma_\delta$
to acquire the prefix ``quasi'', or eliminates $V$ by moving it to
$V'$ with an analogous effect. This makes the notation
$\delta'\to\delta$ and $\delta\to\delta'$ chosen by us
self-explanatory. It is obvious that Definition
\ref{def_quasigraph} gives rise to two different types of
quasiedges, depending on the type of quasivertex these are
incident to. Since $\Gamma_\delta$ by assumption only contains
edges of mixed type, a quasiedge $e$ incident to a quasivertex of
$\delta'\to\delta$ type must be also incident to a vertex of
$\delta$ type; a quasiedge $e$ incident to a quasivertex of
$\delta\to\delta'$ type must be also incident to a vertex of
$\delta'$ type. We will further need the following

\begin{defn}\label{def_quasiM}
An $M$-matrix $\hat M^{(e)}(\mu)$ of a quasigraph
$\hat\Gamma_\delta^{(e)}$ is defined as follows (cf.
\eqref{Eq_Weyl_Func_Delta} and \eqref{Eq_Weyl_Func_Delta'}):

(i) If $V_k$ is a vertex of $\delta$ type or a quasivertex of
$\delta'\to\delta$ type, $\hat m_{jk}(\mu)=$
\begin{equation}\label{Eq_Weyl_Func_Delta_quasi}
\begin{cases}\scriptsize
\mu\Bigl(\sum_{e_t\in
\hat E'_k}\tan\mu l_t+ 2\sum_{e_t\in L_k}\tan\frac{\mu l_t}{2}   \\
-2\sum_{e_t\in \hat L_k}\cot\left(\frac\pi 4+\frac{\mu
l_t}{2}\right)-2\sum_{e_t\in \hat L'_k}\cot\frac{\mu
l_t}{2}\Bigr),
            & j=k,\\
-\mu\sum_{e_t\in \hat C_{kj}}\frac{1}{\cos\mu l_t} -\sum_{e_t\in
C'_{kj}}\frac{1}{\cos\mu l_t},
            & j\not=k,\  V_j \mbox{ is  }\\
            & \mbox{ adjacent to}\ V_k,\\
            0,
            & j\not=k,\  V_j \mbox{ is  }\\
            & \mbox{ not adjacent to } V_k;\\
     \end{cases}
\end{equation}

(ii) If $V_k$ is a vertex of $\delta'$ type or a quasivertex of
$\delta\to\delta'$ type, $\hat m_{jk}(\mu)=$
\begin{equation}\label{Eq_Weyl_Func_Delta'_quasi}
\begin{cases}
\frac{1}{\mu}\Bigl(\sum_{e_t\in \hat E'_k}\tan\mu
l_t-2\sum_{e_t\in L_k}\cot\frac{\mu
 l_t}{2}\\
-2\sum_{e_t\in \hat L_k}\cot\left(\frac\pi 4+\frac{\mu
l_t}{2}\right)+2\sum_{e_t\in \hat L'_k}\tan\frac{\mu
l_t}{2}\Bigr),  & j=k,\\
-\sum_{e_t\in C'_{kj}}\frac{1}{\cos\mu
l_t}-\frac{1}{\mu}\sum_{e_t\in \hat C_{kj}}\frac{1}{\cos\mu l_t},
            & j\not=k,\  V_j \mbox{ is  }\\
            & \mbox{  adjacent to }\ V_k,\\
            0,
            & j\not=k,\  V_j \mbox{ is }\\
            & \mbox{not adjacent to } V_k.\\
     \end{cases}
\end{equation}
Here $L_k$ is the set of loops at the (quasi)vertex $V_k,$ $\hat
L_k$ is the set of quasiloops at $V_k$ coming from elimination of
multiple edges, $\hat L'_k$ is the set of quasiloops at $V_k$
coming from loops at the eliminated vertex, $\hat E'_k$ is the set
of graph (quasi)edges incident to the (quasi)vertex $V_k$, $\hat
C_{kj}$ is the set of quasiedges connecting (quasi)vertex $V_k$
with $V_j$, and finally, $C'_{kj}$ is the set of graph edges
connecting (quasi)vertex $V_k$ with $V_j$. Surely, for any pair of
adjacent $V_k$, $V_j\in\hat\Gamma_\delta^{(e)}$ one either has
$\hat C_{kj}=\emptyset$ or $C'_{kj}=\emptyset$ since all
connections between them are either edges or quasiedges by
construction.
\end{defn}

We remark that for all $k$ and $j$ (including $k=j$) such that
$V_k$ and $V_j$ are ``normal'' graph vertices this definition is
nothing but \eqref{Eq_Weyl_Func_Delta} and
\eqref{Eq_Weyl_Func_Delta'}, where we have used the assumption
that all  graph edges are of mixed type. As for the situation when
either $V_k$ or $V_j$ is in fact a quasivertex, one notes that
quasiedges have (in terms of powers of $\mu$) a similar effect to
regular edges and loops of $\delta-\delta$ and $\delta'-\delta'$
types.

From the analytic point of view however the object just defined is
drastically different to the Weyl-Titchmarsh $M$-matrix of the
graph $\Gamma_\delta$ since it is not meromorphic in the $\lambda$
plane. It is nonetheless meromorphic on the Riemann surface of
$\sqrt{\lambda}$ (or, in other words, in $\mu$ plane), which is
sufficient for our needs.

One has the following

\begin{lem}\label{Lemma_removal_of_difficult_edge}
Assume that $\Gamma_\delta$ is a graph of more than two vertices
with all edges (except loops) of ``mixed'' type. Assume that
$A_{\vec\alpha}$ and $A_{\vec{\tilde\alpha}}$ are two isospectral
graph Laplacians on $\Gamma_\delta$. Let $V$ and $V'$ be $\delta$
and $\delta'$ endpoints, respectively, of an edge $e$. Finally let
$\hat \Gamma_\delta^{(e)}$ be constructed by either of two ways
described in Definition \ref{def_quasigraph}. Then the identity
\begin{equation}\label{eq_removal1}
\Phi(\vec\alpha)\det(\hat M^{(e)}(\mu)-\hat
B_e)=\Phi(\vec{\tilde\alpha}) \det(\hat M^{(e)}(\mu)-\hat{\tilde{
B_e}})
\end{equation}
holds for all $\mu\in\mathbb C$ except for the countable set of
poles of $\hat M^{(e)}(\mu)$. Here $\hat M^{(e)}$ is the
$M$-matrix of the quasigraph $\hat\Gamma_\delta^{(e)}$, whereas
$\hat B_e$ and $\hat{\tilde {B_e}}$ are diagonal $(N-1)\times
(N-1)$ matrices of coupling constants $\vec\beta$ and
$\vec{\tilde\beta}$, respectively, associated with all
(quasi)vertices of $\hat\Gamma_\delta$ and defined as follows:

(i) \emph{(if one removes $e$ to form a quasivertex $\hat V$ of
$\delta'\to\delta$ type)} $\beta_V=\alpha_V$,
$\tilde{\beta_V}=\tilde{\alpha_V}$ for any vertex $V$ of
$\hat\Gamma_\delta^{(e)}$ and $\beta_{\hat V}=\alpha_V+\mu^2
\alpha_{V'}$, $\widetilde{{\beta_{\hat V}}}=\tilde\alpha_V+\mu^2
\tilde\alpha_{V'}$ for the quasivertex $\hat V$;

(ii) \emph{(if one removes $e$ to form a quasivertex $\hat V$ of
$\delta\to\delta'$ type)} $\beta_V=\alpha_V$,
$\tilde{\beta_V}=\tilde{\alpha_V}$ for any vertex $V$ of
$\hat\Gamma_\delta^{(e)}$ and $\beta_{\hat V}=\alpha_{V'}+\mu^{-2}
\alpha_{V}$, $\widetilde{\beta_{\hat
V}}=\tilde\alpha_{V'}+\mu^{-2} \tilde\alpha_{V}$ for the
quasivertex $\hat V$.
\end{lem}

The \emph{proof} of this Lemma is nothing but a slight
modification of the proof of Theorem \ref{thm_trimming}. We
therefore take the liberty of omitting it.

Noting that under the additional assumption of either
$\alpha_{V'}=\tilde\alpha_{V'}=0$ or $\alpha_V=\tilde\alpha_V=0$
either the matrix $\hat B_e$ or the matrix $\hat{\tilde {B_e}}$
does not depend on $\mu^2$, one arrives at

\begin{thm}\label{Thm_00}
Assume that $\Gamma_\delta$ is a graph of more than 2 vertices
with all edges (except loops) of mixed type. Assume that
$A_{\vec\alpha}$ and $A_{\vec{\tilde\alpha}}$ are two isospectral
graph Laplacians on $\Gamma_\delta$. Let $V$ and $V'$ be $\delta$
and $\delta'$ endpoints, respectively, of an edge $e$. Then the
graph Laplacians $A_{\vec{\alpha_e}}$ and
$A_{\vec{\tilde{\alpha_e}}}$ on the graph $\Gamma_\delta^{(e)}$
are isospectral for $\vec\alpha_e$ and $\vec{\tilde{\alpha_e}}$.
 Here

(i) if $\alpha_{V'}=\tilde\alpha_{V'}=0$, the graph
$\Gamma_\delta^{(e)}$ is nothing but $\hat\Gamma_\delta^{(e)}$ of
Definition \ref{def_quasigraph}, part (i), where the quasivertex
$\hat V$ of $\delta'\to\delta$ type is converted to a regular one
by assigning the type $\delta$ to it with the coupling constant
taken from the vertex $V$. $\vec\alpha_e$ and
$\vec{\tilde{\alpha_e}}$ are constructed by dropping $\alpha_{V'}$
and $\tilde{\alpha_{V'}}$ from $\vec\alpha$ and
$\vec{\tilde\alpha}$, respectively;

(ii) if $\alpha_{V}=\tilde\alpha_{V}=0$, the graph
$\Gamma_\delta^{(e)}$ is nothing but $\hat\Gamma_\delta^{(e)}$ of
Definition \ref{def_quasigraph}, part (ii), where the quasivertex
$\hat V$ of $\delta\to\delta'$ type is converted to a regular one
by assigning the type $\delta'$ to it with the coupling constant
taken from the vertex $V'$. $\vec\alpha_e$ and
$\vec{\tilde{\alpha_e}}$ are constructed by dropping $\alpha_{V}$
and $\tilde{\alpha_{V}}$ from $\vec\alpha$ and
$\vec{\tilde\alpha}$, respectively.
\end{thm}

\begin{proof}[A sketch of the proof]
It suffices to check the conditions of ``if'' implication of
Theorem \ref{Thm_n_and_s}, i.e., to ascertain \eqref{eq_n_and_s}.

On the other hand, Lemma \ref{Lemma_removal_of_difficult_edge}
yields \eqref{eq_removal1}, and moreover within our assumptions
one clearly has $\Phi(\vec\alpha)=\Phi_e(\vec{\alpha_e})$,
$\Phi(\vec{\tilde\alpha})=\Phi_e(\vec{\tilde{\alpha_e}})$.

Repeating the argument leading to Theorem \ref{Thm_det} using
Definition \ref{def_quasiM}, one by the linear independence
argument, exactly as in the proof of Theorem \ref{Thm_n_and_s} (in
the process one surely has to redefine the function $\Pi(\mu)$
appropriately, see Lemma \ref{Lemma_entire} for details),
immediately obtains \eqref{eq_n_and_s}. The only part of this
argument that deserves a comment is the fact that the cancellation
procedure described in Definition \ref{def_cancellation} is
exactly the same for $\hat \Gamma_\delta^{(e)}$ as it would be for
$\Gamma_\delta^{(e)}$. This follows immediately from Proposition
\ref{Prop_M} and Definition \ref{def_quasiM}.
\end{proof}

Unfortunately, a straightforward application of this Theorem is
not enough to complete the job due to the possible presence of
$0\bar0$ and $\bar00$ class vertices. This becomes apparent if one
considers a star-graph of exactly 3 rays, one of the boundary
vertices being of $00$ class. The argument presented then reduces
the problem to $A_3$ which leaves the question of isospectrality
open for the named star graph. In fact, a direct application of
Theorem \ref{Thm_n_and_s} yields uniqueness in this case, thus
necessitating further in-depth analysis of the general situation.

First we employ the results of \cite{Yorzh3} which together with
Theorem \ref{Thm_00} and Theorem \ref{thm_uniqueness_nonzero}
allow to give a complete solution of the problem of isospectrality
provided that there are no vertices of $0\bar0$ and $\bar00$
classes.

\begin{thm}\label{Thm_uniqueness_zero}
Assume that $\Gamma_\delta$ is an arbitrary marked graph of more
than two vertices. Assume that $A_{\vec\alpha}$ and
$A_{\vec{\tilde\alpha}}$ are two isospectral graph Laplacians on
$\Gamma_\delta$. Moreover, assume that there are no vertices $V$
of valence 2 satisfying the condition $\alpha_V=\tilde\alpha_V=0$.
Then the following implication holds: if at a vertex $V$ one has
$\alpha_V=\tilde\alpha_V=0$, then for every graph vertex $W$
adjacent to $V$ the identity $\alpha_W=\tilde\alpha_W$ holds.
\end{thm}

\begin{proof}
Let $V$ be a vertex of $00$ class. Theorem \ref{thm_trimming} and
Theorem \ref{Thm_00} justify the procedure of removing any edge
$e$ incident to this vertex, be it of $\delta-\delta$,
$\delta'-\delta'$ or $\delta-\delta'$ type, with absolutely the
same implications on the graph itself: the vertex $V$ joins $W$,
the other endpoint of $e$, whereas the type of $W$ and the
coupling constants $\alpha_W$ and $\tilde\alpha_W$ remain
unchanged. The corresponding graph Laplacians $A_{\vec{\alpha_e}}$
and $A_{\vec{\tilde{\alpha_e}}}$ are isospectral.

On the other hand, the results of \cite{Yorzh3} tell us that under
the condition of isospectrality vectors $\vec\alpha$ and
$\vec{\tilde\alpha}$ have the same number of zero components;
what's more, equal numbers of zero coupling constants pertaining
to vertices of $\delta'$ type (and hence of $\delta$ type) are
also guaranteed. Moreover, as in the proof of Theorem
\ref{thm_uniqueness_nonzero}, one has the set equality
$S(\vec\alpha)=S(\vec{\tilde\alpha})$ (see \eqref{eq_sigma} for
the definition of the set $S$). In our setting, the same number of
sigmas, say, $k$ is equal to zero in both $S$-sets. Let $\hat S$
be the set $S$, where all zero elements have been removed. Then
under the condition of isospectrality $\hat S(\vec\alpha)$ is
nothing but a transposition of the set $\hat
S(\vec{\tilde\alpha})$.

It remains to be seen that $\sigma_W=\tilde\sigma_W$. Assume the
opposite. Note that the procedure of removing the edge $e$ changes
the valence of $W$, since $V$ is by assumption not of valence 2.
It turns out that $\hat S(\vec\alpha_e)$ and $\hat
S(\vec{\tilde{\alpha_e}})$ are the same sets as $\hat
S(\vec\alpha)$ and $\hat S(\vec{\tilde \alpha})$, respectively,
with the only exception: the elements $\sigma_W$ and
$\tilde\sigma_W$ are replaced by $\kappa\sigma_W$ and
$\kappa\tilde\sigma_W$ with $\kappa\neq 1$, respectively. The
constant $\kappa$ here depends on the valences of $V$ and $W$
only.

On the other hand, $\hat S(\vec\alpha_e)$ and $\hat
S(\vec{\tilde{\alpha_e}})$ have to be transpositions of each other
by isospectrality of $A_{\vec{\alpha_e}}$ and
$A_{\vec{\tilde{\alpha_e}}}$. One easily convinces oneself that
this immediately leads to a contradiction, unless
$\sigma_W=\tilde\sigma_W$.
\end{proof}

\begin{rem}\label{cleansing}
The condition of the latter Theorem that there are no $00$ class
vertices of valence 2 in $\Gamma_\delta$ can be assumed w.l.o.g.
Indeed, if one has a $00$ vertex in $\Gamma_\delta$, either of
$\delta$ or $\delta'$ type, it follows that a function in the
domain of both $A_{\vec\alpha}$ and $A_{\vec{\tilde\alpha}}$ is
continuous together with its first derivative through the given
vertex. This means that the vertex itself can be removed from the
graph (cf. \cite{Aharonov}, where this procedure is called
\emph{cleaning}).
\end{rem}

It is now left to complete our analysis of graphs $\Gamma_\delta$
with all edges of mixed type. We will consider the situation when
there exist vertices of $0\bar0$ and/or $\bar00$ class.

The analysis we are about to present is probably not the simplest
way of obtaining the result sought, but in our view the most
transparent one. It is based on Lemma
\ref{Lemma_removal_of_difficult_edge}. As above, we assume
w.l.o.g. that $\Gamma_\delta$ contains no vertices of valence 2
which are of $00$ class.

We start by ascertaining the following corollary of the named
Lemma. Assume that a vertex $V'$ is of $\delta'$ type and of
$0\bar0$ class. Then a direct application of Lemma
\ref{Lemma_removal_of_difficult_edge} yields: any other $\delta'$
vertex $W'$ of the graph $\Gamma_\delta$ which is
$\delta$-adjacent to $V'$ (i.e., is adjacent to a $\delta$ type
vertex which is in turn adjacent to $V'$ or, in other words, is
located within a single $\delta$ type vertex from $V'$) is of
either $\bar00$ or $00$ class. Indeed, let $V'-W-W'$ form a chain
of $\delta'-\delta-\delta'$ vertices. Removing an edge between
$V'$ and $W$ to form a quasivertex $\hat V$ of $\delta'\to\delta$
type, use the part (i) of Lemma
\ref{Lemma_removal_of_difficult_edge} and compare the terms on
both sides of \eqref{eq_removal1} which involve the highest power
of $\mu$. These terms have to be equal
 due to the linear independence argument. Noting that
$\hat B_e$ is independent of $\mu$, whereas $\hat{\tilde{B_e}}$
contains the term $\tilde{\alpha_{V'}} \mu^2$ in the coupling
constant pertaining to $\hat V$, decompose the determinant on the
right hand side in the following way:
\begin{equation}\label{eq_det_decomposition}
\det(\hat M^{(e)}(\mu)-\hat{\tilde{ B_e}})=\det(\hat
M^{(e)}(\mu)-B_0)+\det(\hat M^{(e)}_V(\mu)-B_{V}),
\end{equation}
where the matrices $B_0$ and $B_{V}$ differ only in the coupling
constant pertaining to the quasivertex $\hat V$, this element
being equal to $\tilde\alpha_W$ and $\tilde{\alpha_{V'}} \mu^2$,
respectively (cf. Lemma \ref{Lemma_removal_of_difficult_edge}),
and $\hat M^{(e)}_V(\mu)$ is $\hat M^{(e)}(\mu)$ with both the
$\hat V$th row and column zeroed out.

From Theorem \ref{Thm_n_and_s} one knows that there must be a
vertex of $\delta'$ type and of $\bar 00$ class in
$\Gamma_\delta$. It then becomes apparent (cf. Definition
\ref{def_quasiM}) that the terms with the highest power in $\mu$
come from the second term on the r.h.s. of
\eqref{eq_det_decomposition} \emph{only}.

On the other hand, by Theorem \ref{Thm_det} (cf. the proof of
Theorem \ref{Thm_00}), one can single out the following subgraphs
$G\in\mathcal G$ which are sure to yield the topmost power of
$\mu$ on the l.h.s. of \eqref{eq_removal1}: pick an alpha-loop at
every $\delta'$ type vertex with non-zero coupling constant and a
regular loop at any other vertex of the graph
$\Gamma_\delta^{(mod)}$ (one uses Lemma \ref{Lemma_marking} to
ensure that subgraphs selected are in fact in
$\hat{\mathfrak{G}}$). The second term on the r.h.s. of
\eqref{eq_det_decomposition} can in turn be treated as the sum of
subgraph weights over subgraphs of $\mathcal G$ such that the
alpha-loop with the weight $\tilde{\alpha_{V'}} \mu^2$ is taken at
the vertex $\hat V$. It is now obvious that if
$\tilde\alpha_{W'}\neq 0$, there exists no subgraph yielding the
highest power of $\mu$ \emph{and} containing $\tan l_{WW'}\mu$,
where $l_{WW'}$ is the length of the edge $WW'$, on the r.h.s. of
\eqref{eq_removal1}, whereas this term exists on the l.h.s. due to
Lemma \ref{Lemma_marking}, \emph{unless} $\Gamma_\delta$ is $A_3$.
In the $A_3$ case however Theorem \ref{Thm_n_and_s} already
guarantees that $W'$ must be of $\bar00$ class. We have therefore
arrived at a contradiction.

Consideration of a vertex $V$ of $\delta'$ type and of $0\bar0$
class is done in absolutely the same way, comparing the lowest
possible powers of $\mu$. One then obtains that any other $\delta$
vertex $W$ of the graph $\Gamma_\delta$ which is
$\delta'$-adjacent to $V$ is of either $\bar00$ or $00$ class.

One can in fact do even better than that. We will argue that in
the setup introduced above the vertex $W'$ cannot be of $00$ class
provided that $W'\in\partial \Gamma_\delta$. Indeed, having
assumed the opposite, one sees by means of the same argument that
the weight of \emph{any} subgraph of $\mathcal G$ giving rise to
the highest power in $\mu$ on the r.h.s. of \eqref{eq_removal1}
\emph{must} contain $\tan l_{WW'}\mu$, whereas there clearly
exists a subgraph on the l.h.s. such that it does not. This
subgraph is nothing but the 2-loop of vertices $W$ and $W'$, built
up to a spanning subgraph of $\mathcal G$ in a way such that its
weight contains the maximal power of $\mu$ (i.e., by taking
alpha-loops at all graph vertices of $\delta'$ type with non-zero
coupling constants). By linear independence argument this leads to
a contradiction. Surely the same holds for a boundary vertex $W$
of $\delta$ type and of $00$ class.

It is then possible, having started with a graph $\Gamma_\delta$
having every edge of mixed type, to proceed as follows.

If there exists a $\delta'$ vertex $V'$ of $0\bar0$ class, start
with it. If there are no such vertices, find a $\delta$ vertex $V$
of the same class instead. If this also fails, one uses Theorem
\ref{Thm_uniqueness_zero}.

Starting with $V'$, consider all $\delta$-adjacent vertices. Those
of them that are internal might be of $00$ class; in this case
they are removed using Theorem \ref{Thm_00}. After this first
removal, a certain number of $\delta-\delta$ edges appears in the
graph, which we remove using Theorem \ref{thm_trimming}. The
procedure eventually leads to a graph $\Gamma_\delta^{(0)}$ which
has all $\delta'$ vertices of $0\bar0$ or $\bar00$ classes and by
construction has all its edges of mixed type.

Now consider every vertex $V$ of $\delta$ type to be found in the
graph $\Gamma_\delta^{(0)}$. Our first claim is that in order to
avoid a contradiction, it might be connected to at most 2
different vertices (of $\delta'$ type). This follows from the fact
obtained above that $\delta$-adjacent $\delta'$ type vertices have
to be of \emph{different classes}, i.e., if one of them is
$0\bar0$, the other one must be $\bar00$, in conjunction with the
absence of $00$ class vertices of $\delta'$ type. Moreover, a
simple argument based again on Lemma
\ref{Lemma_removal_of_difficult_edge} shows, that the valence of
$V$ must be equal to 2 (thus, multiple edges incident to it are
disallowed).

Therefore, unless we have already arrived at a contradiction, we
have a graph $\Gamma_\delta^{(0)}$ with all $\delta'$ vertices of
either $0\bar0$ or $\bar00$ class, ``connected'' by $\delta$
vertices of valence 2.

This graph must contain at least a single internal $\delta'$
vertex (otherwise the graph $\Gamma_\delta^{(0)}$ reduces to the
chain $A_3$, but then clearly $\Gamma_\delta$ was a tree; then if
$\Gamma_\delta\neq A_3$, at the penultimate step of reduction one
can apply Theorem \ref{Thm_00} to show that
$\alpha_V=\tilde\alpha_V$ at the only internal vertex of $\delta$
type, which prevents isospectrality by Example \ref{A3}). There
are two distinct possibilities: one either has no $\delta$ type
vertices of $0\bar0$ class or there exists at least one.

In the former case, one can remove any of $00$ classed $\delta$
vertices. We note that these vertices cannot be located on the
graph boundary since otherwise Theorem \ref{Thm_uniqueness_zero}
is applicable, leading to a contradiction. Removing an internal
$00$ vertex, one gets at least one edge $e$ connecting two
vertices of $\delta'$ type and of $0\bar0$ and $\bar00$ classes,
respectively. Theorem \ref{thm_trimming} yields that removing this
edge one gets a $\bar0\bar0$ classed $\delta'$ vertex which leads
to a contradiction by the argument above, provided that
$\Gamma_\delta^{(0)}$ had more than 2 vertices of $\delta'$ type.
If on the other hand there were just 2 in $\Gamma_\delta^{(0)}$,
one has to resort to Lemma \ref{Lemma_removal_of_difficult_edge}.
A straightforward but rather lengthy argument then leads to a
contradiction again.

Thus, one may assume that the graph $\Gamma_\delta^{(0)}$ contains
no $00$ classed vertices of $\delta$ type. One then applies the
analysis developed in the proof of Theorem
\ref{thm_uniqueness_nonzero} verbatim for an \emph{internal}
vertex of $\delta'$ type. The corresponding balancing equation
immediately yields a contradiction.

If on the other hand the graph does contain a $\delta$ type vertex
of $0\bar0$ class (and then another one of $\bar00$ class), this
again yields that any $\delta'$-adjacent $\delta$ type vertex is
of either $\bar00$ or $00$ class (the latter possibility being
disallowed for boundary vertices). The possibility of internal
$00$ class $\delta$ type vertices also leads to a contradiction.
Assume the opposite: let $V$ be a $\delta$ type vertex of valence
2 and of $00$ class. One first removes it by Remark
\ref{cleansing} and then uses Theorem \ref{thm_trimming} to get
rid of the resulting $\delta'-\delta'$ edge, giving rise to a
$\delta'$ type vertex of $\bar0\bar0$ class. If one has at least
one more $\delta$-adjacent $\delta'$ vertex of $\bar00$ or
$0\bar0$ class, this turns out to be impossible. In the remaining
case of \emph{just} two $\delta'$ vertices in
$\Gamma_\delta^{(0)}$ linked together by a $00$ class $\delta$
vertex, one gets to a contradiction by applying Lemma
\ref{Lemma_removal_of_difficult_edge}.

The graph $\Gamma_\delta^{(00)}$ thus obtained is clearly either a
chain or a simple cycle. In both cases, it has (i) all edges of
mixed type, and (ii) all vertices of either $0\bar0$ or $\bar00$
class.

In the case when $\Gamma_\delta^{(00)}$ is a chain, a
straightforward check based on Lemma
\ref{Lemma_removal_of_difficult_edge} again ensures a
contradiction.

In the case of a simple cycle (surely, conditions (i) and (ii)
tell us that this has to be a cycle of $4k$ vertices for natural
$k$) one has to use Theorem \ref{Thm_n_and_s}. Since this graph is
not a tree, Corollary \ref{cor_n_and_s} is applicable. A rather
straightforward analysis now shows, that in this case non-trivial
isospectral configurations are disallowed.

Ultimately, we sum up results obtained above in the form of the
following
\begin{thm}\label{Thm_bloody}
Assume that $\Gamma_\delta$ is a graph which is not a chain $A_2$
of exactly 2 vertices with each edge (except loops) of mixed type.
Assume that $A_{\vec\alpha}$ and $A_{\vec{\tilde\alpha}}$ are two
isospectral graph Laplacians on $\Gamma_\delta$. If the graph
contains no vertices $V$ of valence 2 such that\footnote{This
assumption is w.l.o.g. by Remark \ref{cleansing}.}
$\alpha_V=\tilde\alpha_V=0$, then either $\Gamma_\delta=A_3$ of
Example \ref{A3} \emph{or} $\vec\alpha=\vec{\tilde\alpha}$.
\end{thm}

\subsection*{Acknowledgements} The authors express deep gratitude
to Prof. Sergey Naboko  for constant attention to their work. We
would also like to cordially thank our referees for making some
very helpful comments.

\end{document}